\documentclass[UTF-8,reqno,12pt]{amsart}
\usepackage{enumerate}
\usepackage{mhequ}
\usepackage{amsmath}
\usepackage{siunitx}
\usepackage{enumitem} 
\usepackage{bbm}
\usepackage{geometry}
\usepackage[normalem]{ulem}
\usepackage{amssymb,url,color, booktabs,nccmath}
\usepackage{mathrsfs}
\usepackage{graphicx}
\usepackage{tikz}
\usetikzlibrary{arrows.meta, positioning, shapes.geometric, calc}
\usepackage{amsthm}
\usepackage{comment}
\usepackage{relsize}
\usepackage{algorithm}
\usepackage{algpseudocode}
\usepackage{caption}
\usepackage{placeins} 
\usepackage{float} 
\usepackage{enumerate}

\usepackage{dutchcal}
\usepackage{pifont}
\usepackage[colorlinks=true]{hyperref}
\hypersetup{
	linkcolor=blue,
	citecolor=red,
	filecolor=blue,
	urlcolor=cyan
}

\makeatletter
	\renewcommand{\subsection}{\@startsection
		{subsection}
		{2}
		{0mm}
		{0.5\baselineskip}
		{0.3\baselineskip}
		{\normalfont\normalsize\raggedright}}
\makeatother
\allowdisplaybreaks[4]
\usepackage{fancyhdr}
\numberwithin{equation}{section}

\def\dif{\mathrm{d}}
\def\E{\mathbb{E}}

\def\R{\mathbb{R}}

\def\d{\mathrm{d}}
\def\N{\mathbb{N}}

\def\1{\mathbbm{1}}
\def\e{\mathrm{e}}

\def\geq{\geqslant}
\def\leq{\leqslant}
\def\ge{\geqslant}
\def\le{\leqslant}

\newtheorem{theorem}{Theorem}[section]
\newtheorem{lemma}[theorem]{Lemma}
\newtheorem{remark}[theorem]{Remark}

\newtheorem*{theorem*}{Theorem}
\newtheorem*{remark*}{Remark}

\begin{document}
	
	\title{Numerical Approximation for Stochastic differential equations with State-Dependent  fast Switching}
	
\author{Xiaobin Sun\quad }
\curraddr[Sun, X.]{School of Mathematics and Statistics/RIMS, Jiangsu Normal University, Xuzhou, 221116, P.R. China}
\email{xbsun@jsnu.edu.cn}

\author{Mingkun Ye$^{*}$}
\curraddr[Ye, M.]{School of Mathematics and Statistics, Wuhan University, Wuhan, 430072, P.R. China}
\email{mingkunye@foxmail.com}
	
\author{Zuozheng Zhang}
\curraddr[Zhang, Z.]{School of Mathematics and Statistics, Wuhan University, Wuhan, 430072, P.R. China}
\email{zuozhengzhang@mail.bnu.edu.cn}
	
	\begin{abstract}
 	This paper aims to develop efficient numerical approximations for a class of stochastic differential equations with state-dependent fast switching processes. The direct Euler--Maruyama (EM) scheme fails when the scaling parameter is small. Based on the heterogeneous multiscale method of \cite{e2005analysis}, we propose three algorithms and prove their strong $L^p$-convergence with explicit rates for any $p\geq 2$. In the first algorithm, we combine the averaging principle with an EM scheme for the averaged equation, where the invariant measure of the Markov chain can be explicitly obtained by solving a linear system. However, computing this invariant measure incurs cubic cost as the number of switching states increases.

    To avoid solving large linear systems, we approximate the invariant measure instead. Thus the second and third algorithms both combine a macroscopic EM scheme for a modified averaged equation with micro-solvers that estimate the averaged drift. More precisely, in the second algorithm, a discrete-time Markov chain is simulated with a micro time step, and the averaged drift is obtained by averaging over finitely many micro transitions. However, both the first and second algorithms only work when the switching process has finite states. Therefore, we introduce a third algorithm that allows for switching processes with countably infinite states, in which exact continuous-time Markov chain is generated via the Gillespie algorithm, and the averaged drift is computed by exact time averaging over a specified interval. Numerical experiments verify the theoretical results and demonstrate the computational advantages of these three algorithms.
    \\
	\end{abstract}
	\thanks{$*$ Corresponding author}
	\maketitle

    \tableofcontents


\section{Introduction}
Stochastic differential equations (SDEs) with state-dependent switching provide a flexible framework for modeling abrupt structural shifts and random environmental regimes across finance, ecology, and control engineering. Specifically, this kind of model comprises two components \((X_t, \Lambda_t)\): \(X_t\) describes the spatial location of the system at time $t$, \(\Lambda_t\) indicates the system’s current regime. Its dynamics are given by the following SDEs:
\[
\begin{cases}
\dif X_t = b(X_t,\Lambda_t)\dif t + \sigma(X_t,\Lambda_t)\dif W_t,\\
\mathbb{P}\{\Lambda_{t+\Delta}=j|\Lambda_t=i,X_s,\Lambda_s,s\le t\}=
\begin{cases}
q_{ij}(X_t)\Delta+o(\Delta), & j\neq i\\
1+q_{ii}(X_t)\Delta+o(\Delta), & j=i,
\end{cases}\\
(X_0,\Lambda_0)=(x_0,i_0)\in\mathbb{R}^n\times\mathbb{S},
\end{cases}
\]
where $W_t$ is a standard $d$-dimensional Brownian motion on $(\Omega,\mathscr{F},\mathbb{P})$ with natural filtration $\{\mathscr{F}_t\}_{t\ge0}$,  $b\colon\mathbb{R}^n\times \mathbb{S}\to\mathbb{R}^n$, $\sigma\colon\mathbb{R}^n\times \mathbb{S}\to\mathbb{R}^n\otimes\mathbb{R}^d$, and $\mathbb{S}=\{1,2,\dots,N\}$ with $N\leq\infty$. When $q_{ij}(x)$ is independent of $x$ for all $i,j\in \mathbb{S}$, the above system is called SDEs with  state‑independent switching or SDEs with Markovian switching.

As explicit solutions are rarely available for such systems, numerical discretization serves as the primary tool for both theoretical analysis and practical computation. The systematic study of EM scheme for SDEs with Markovian switching was initiated by the seminal work of \cite{YUANandMAO2004MCS}, with several early follow-up works further developing the corresponding numerical theory under the global Lipschitz condition \cite{MaoYuanYin2005-JCAM,YuanMao2005-JDEA}.
Subsequent work has extended the convergence theory of the EM scheme to increasingly relaxed regularity conditions \cite{YuanMao2007-JCAM,NguyenNguyen2019-CoSA}. Further related results can be found in the monograph \cite{mao2006stochastic}. While most early studies focused on state-independent switching, the monograph \cite{YINandZHU2010BOOK} provides a comprehensive treatment of strong approximation, martingale-problem-based weak convergence, and numerical methods for the more technically challenging state-dependent case. \cite{JinShenSu2025-JTP} recently establishs strong \(L^1\) and \(L^2\) convergence of the EM scheme with explicit error bounds for the more technically demanding state-dependent case under non-Lipschitz coefficients. For superlinear switching SDEs where standard explicit EM fails due to moment explosion, \cite{LiMaYang2018-SINUM} employs the backward EM method to study numerical invariant measures. To circumvent the high computational cost of implicit schemes, \cite{NguyenNguyenYin2021-AMC} proposes a tamed EM scheme for superlinear hybrid SDEs with strong convergence proved under local Lipschitz and Khasminskii-type growth conditions. None of the aforementioned works consider fast switching scenarios. 

SDEs with fast switching can characterize a broad class of real-world phenomena, where structural and environmental transitions evolve on a much faster time scale than continuous diffusion dynamics, see neuronal model in \cite[Section 3]{PTW2012}.  Now let us consider a simple example on $\mathbb{R}\times \{1,2\}$:
\begin{equation}\label{Eq:fcz}
	\begin{cases}
		\d X_t = b(X_t, \Lambda_t) \d t + (1/5)\d W_t, \\
        (X_0, \Lambda_0) = (1, 1) \in \mathbb{R} \times \{1,2\},
	\end{cases}
\end{equation}
where the drift coefficient $b$ takes two distinct linear forms: $$b(x,1)=- x,\quad b(x,2)=-2x + 3 $$ and $\Lambda_t$ is a two states $\{1,2\}$ Markov chain with generator 
\[
Q^{\varepsilon} = \varepsilon^{-1}\begin{pmatrix}
	-{1}/{5} & {1}/{5} \\
	{2}/{5} & -{2}/{5}
\end{pmatrix}, \quad \varepsilon>0.
\]
 It follows from  \cite{mao2006stochastic} that the system \eqref{Eq:fcz} admits a strong solution. Then, consider a finite-state discrete-time Markov chain (DTMC) \((\alpha_{m})_{m\geq 0}\) with one-step transition probability matrix \(P=(p_{ij})_{N\times N}\). Following the inverse transform method (see, e.g., Mao et al. \cite[Section 3]{MaoYuanYin2005-JCAM}), the DTMC $(\alpha_{m})_{m\geq 0}$ can be simulated as follows:
\begin{algorithm}[htbp]
\caption*{\textbf{Inverse transform method}}
\label{alg:hmm:ITM}
\begin{algorithmic}[1]
\For{$m = 0,1,2,\dots$}
    \State $k \gets \alpha_m$ \Comment{current state}
    \State Extract row $\boldsymbol{p}_k = (p_{k1},p_{k2},\dots,p_{kN})$
    \State Compute $F_k(l) = \sum_{j=1}^{l}p_{kj}$ for $l=1,\dots,N$
    \State Generate $U_m \sim \text{Uniform}(0,1)$
    \State $l^* \gets \min\{l\in\{1,\dots,N\} \mid F_k(l)\ge U_m\}$
    \State $\alpha_{m+1} \gets l^*$
\EndFor
\end{algorithmic}
\end{algorithm}

Based on the above method, we employ the direct EM scheme to stochastic system \eqref{Eq:fcz}, see e.g., \cite{YUANandMAO2004MCS}. Specifically, initialize by setting \((Y_0, \widetilde{\Lambda}_0) = (1, 1)\):
\begin{algorithm}[htbp]
\caption*{\textbf{Direct EM scheme}}
\label{alg:hmm:DEM}
\begin{algorithmic}[1]
\Require  step size $\Delta_1$, terminal time $T$
\For{$n = 0,1,2,\dots,\lfloor T/\Delta_1\rfloor$}
\State Draw Brownian increment $\Delta W_n = W_{(n+1)\Delta_1}-W_{n\Delta_1}$
\State $Y_{(n+1)\Delta_1} \gets Y_{n\Delta_1} + b(Y_{n\Delta_1},\widetilde{\Lambda}_{n\Delta_1})\Delta_1 + \sigma\,\Delta W_n$
    \State $k \gets \widetilde{\Lambda}_{n\Delta_1}$ \Comment{current state}
    \State Compute $P=\e^{\Delta_1 Q^\varepsilon}$
    \State Extract the $k$-th row of $P$: $ \boldsymbol{p}_k = (p_{k1},p_{k2})$  
    \State Compute $F_k(l) = \sum_{j=1}^{l}p_{kj}$ for $l=1,2$
    \State Generate $U_n \sim \text{Uniform}(0,1)$
    \State $l^* \gets \min\{l\in\{1,2\} \mid F_k(l)\ge U_n\}$
    \State $\widetilde{\Lambda}_{(n+1)\Delta_1} \gets l^*$
\EndFor
\end{algorithmic}
\end{algorithm}

 By setting   $ \Delta_1 = 0.01$, $T = 1, $ we performed numerical validation using the EM scheme for this example. 
 As we can see from Figure \ref{fig:direct_euler01}, for relatively large \(\varepsilon\), the direct EM scheme closely matches the true solution. As \(\varepsilon\) decreases, the discrepancy between the exact solution and the  EM scheme increases gradually, and the two curves diverge noticeably. Consequently, the direct EM scheme fails for sufficiently small  \(\varepsilon \), and thus suitable numerical schemes are required to approximate the exact solution $X_t$. 
\begin{figure}[htbp]
	\centering
	\includegraphics[width=1\linewidth]{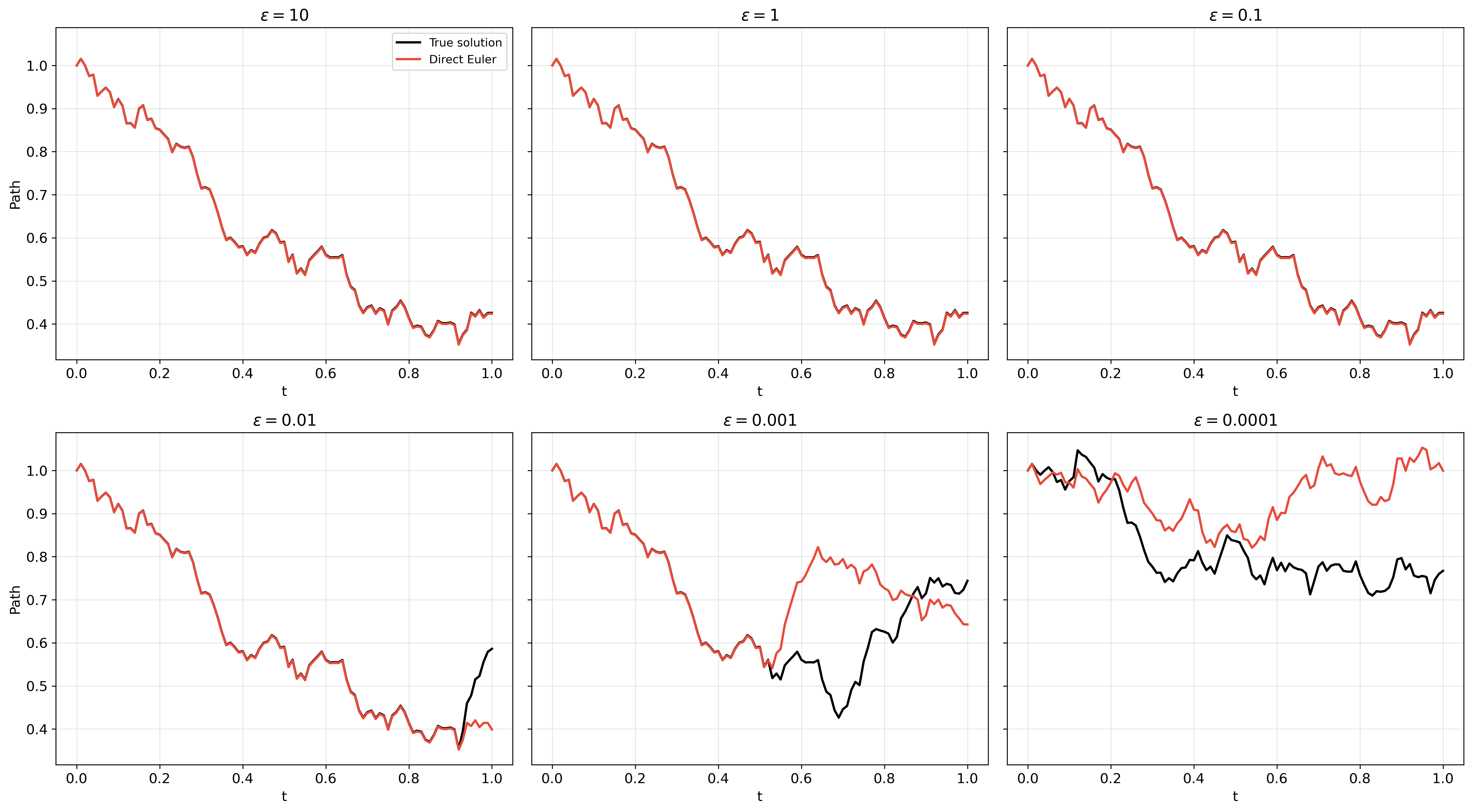}
	\caption{Comparison of sample paths of \(X_t\) and its EM scheme  as \(\varepsilon \downarrow 0\)}
	\label{fig:direct_euler01}
\end{figure}

 In what follows, we consider a more general class of SDEs with state-dependent  fast switching:
\begin{equation}\label{MULTI:EQ:0402:01}
	\left\{\begin{array}{l}
		\d X_{t}^{\varepsilon}=b\left(X_{t}^{\varepsilon}, \Lambda_{t}^{\varepsilon}\right) \d t+\sigma\left(X_{t}^{\varepsilon}\right) \d W_{t} \\
		\mathbb{P}\left(\Lambda_{t+\Delta}^{\varepsilon}=j \mid \Lambda_{t}^{\varepsilon}=i, X_{s}^{\varepsilon}, \Lambda_{s}^{\varepsilon}, s \leq t\right)=\left\{\begin{array}{l}
			\varepsilon^{-1} q_{i j}\left(X_{t}^{\varepsilon}\right) \Delta+o(\Delta), \quad i \neq j, \\
			1+\varepsilon^{-1} q_{i i}\left(X_{t}^{\varepsilon}\right) \Delta+o(\Delta), \quad i=j,
		\end{array}\right. \\
		\left(X_{0}^{\varepsilon}, \Lambda_{0}^{\varepsilon}\right)=(x_0,i_0) \in \mathbb{R}^{n} \times \mathbb{S},
	\end{array}\right.
\end{equation}
 where $\Lambda_t^\varepsilon$ exhibits rapid jumps on a time scale of \( O\left(\varepsilon^{-1}\right) \), with generator \( Q^{\varepsilon}(x)=\varepsilon^{-1} Q(x) \).
For the asymptotic analysis of stochastic system \eqref{MULTI:EQ:0402:01} has been well established, with primary attention devoted to averaging principles \cite{FGR2010,GT2012,MaoandShao2024averaging, PTW2012, SUNandXIE2025EJP,Y2001}, central limit theorem \cite{GT2014,PTW2012, SUNandXIE2025EJP}, diffusion approximations \cite{CAOandWu2026,sun2026diffusion}, and large deviation principles \cite{budhiraja2018large,hu2026large,li2025large,ma2022large}. However, the aim of this paper is the construction of refined numerical schemes for approximating exact solution in the strong sense, ensuring that the numerical solution stays consistent with the exact solution for sufficiently small \(\varepsilon \). It is important to emphasize that the diffusion coefficient takes the form $\sigma(x,y)= \sigma(x)$, which is essential for establishing strong convergence. Without this assumption, the strong averaging principle may fail to hold; see a counter-example in \cite[Remark 2.12]{SUNandXIE2025EJP}.

Note that the stochastic system \eqref{MULTI:EQ:0402:01} can be regard as a slow-fast stochastic system, where the switching term $\Lambda^{\varepsilon}_t$ is the fast component. 
Inspire from a notable work \cite{e2005analysis} for slow-fast coupled SDEs, a very powerful method, named heterogeneous multiscale method (HMM) is proposed to address this issue by coupling a macroscopic integrator with short-time microscopic simulations, yielding accurate approximations at a computational cost largely independent of the small scale parameter; see e.g. \cite{abdulle2012heterogeneous,Br2013,BR2022,cui2023strong,liu2010analysis} for more details.
For  system \eqref{MULTI:EQ:0402:01} under consideration, we now describe our HMM framework in two parts.

 \noindent(I) \textbf{Averaging principle}:
    Our numerical approach is based on the averaging principle for system \eqref{MULTI:EQ:0402:01}. As \(\varepsilon\to0\), the slow component \(X_t^\varepsilon\) converges strongly to the solution \(\bar{X}_t\) of the averaged equation:
    \begin{equation}\label{eq:averaged-equation}
    \dif \bar{X}_t = \bar{b}(\bar{X}_t) \dif t + \sigma(\bar{X}_t) \dif W_t, \quad
    \bar{X}_0 = x_0,
    \end{equation}
    where the averaged drift coefficient is defined by
    \begin{equation}\label{EQ:averaged:drift}
    	\bar{b}(x) := \sum_{i\in\mathbb{S}} b(x,i) \mu_i^{x}.
    \end{equation}
    Here, \(\mu^{x}=(\mu^{x}_i)_{i\in \mathbb{S}}\) denotes the invariant measure of the frozen CTMC \(\Lambda_{t}^{x}\) with generator \(Q(x)\).
    This principle guides our numerical strategy. 

\noindent(II) \textbf{EM scheme to the averaged equation:} Instead of discretizing the slow-fast system \eqref{MULTI:EQ:0402:01} directly, we apply the EM scheme to the averaged equation \eqref{eq:averaged-equation}. Thus the first job is to compute \(\bar{b}(x)\). To do this, we divide into two cases depending on whether the invariant measure \(\mu^x\) of the frozen fast process is explicitly solvable or can be approximated.  
    \begin{enumerate}
        \item [(1)]\textbf{Solvable $\mu^x$}: We can compute its invariant measure exactly by solving the linear system \(\mu^x Q(x)=0\) with the normalization condition \(\mu^x\mathbbm{1}=1\). Consequently, it gives the exact value of \(\bar{b}(x)\). Furthermore, the  standard EM scheme of the averaged equation \eqref{eq:averaged-equation}  is given by
	\begin{equation*}
		Z_{(n+1)\Delta_1}  =   Z_{n\Delta_1} +\bar{b}(Z_{n\Delta_1}) \Delta_1+ \sigma(Z_{n\Delta_1})\Delta W_n,
	\end{equation*}
	where $\Delta_1\in (0,1]$ is the time step size. 
        This algorithm has no sampling error and works very well for small state spaces. But its running time \(O(N^3)\) becomes too high when the number of states \(N\) of the switching process is large; see Section \ref{SEC:Direct-EM} for more details. 
    \item [(2)]\textbf{Approximable $\mu^x$:} In fact, the explicit form of $\mu^x$ is unnecessary, as we may introduce an estimator $\tilde{b}(x)$ to approximate it. Moreover, we require two solvers: a macro solver and a micro solver. The macro solver implements the EM scheme for the slow component, while the micro solver evaluates \(\tilde{b}(x)\) at each macro step and feeds the result back to the macro solver. The details are following:
    \begin{enumerate}
        \item [(2.1)]\textbf{Macro solver}:
 We use the EM scheme to evolve the modified averaged equation described by
	\begin{equation*}
		\hat{Z}_{(n+1)\Delta_1}  =   \hat{Z}_{n\Delta_1} +\tilde{b}(\hat{Z}_{n\Delta_1}) \Delta_1+ \sigma(\hat{Z}_{n\Delta_1})\Delta W_n,
	\end{equation*}
	where $\Delta_1$ is the macro time step size. 
    This macro step size avoids the numerical stiffness caused by the separation of time scales. 

    \item [(2.2)]\textbf{Micro solver}: In total, we give two complementary ways to implement the microscopic step:
    \begin{enumerate}
        \item[(2.2.1)] \textbf{Construction of DTMC}:  At each macro time step, with the slow variable fixed at \(x = \hat{Z}_{n\Delta_1}\), we construct a frozen DTMC \((\Lambda_{m\Delta_2}^{\hat{Z}_{n\Delta_1},i_0})_{m\geq 0}\) via the inverse transform method. The chain has transition probability matrix \(P(x) = \mathrm{e}^{\Delta_2 Q(x)}\) with micro step size \(\Delta_2\), and serves to approximate the fast process. The averaged drift is approximated by taking the arithmetic average over \(M\) consecutive micro steps:
        $$
        \tilde{b}(\hat{Z}_{n\Delta_1})=\tilde{b}_{\Delta_2,M}(\hat{Z}_{n\Delta_1}):= \frac{1}{M}\sum_{m=0}^{M-1} b\left(\hat{Z}_{n\Delta_1}, \Lambda_{m\Delta_2}^{\hat{Z}_{n\Delta_1},i_0}\right),
        $$
        where \(M\) is the total number of micro time steps. This version avoids solving large linear systems, but it adds discretization error from the micro step size \(\Delta_2\) and number $M$; see Section \ref{SEC:CLASSIC} for more details.

        \item [(2.2.2)]\textbf{Construction of CTMC}: As another sampling-based option, we simulate the exact continuous-time path of the frozen CTMC over the interval \([0,R]\) using the Gillespie algorithm. Owing to the piecewise-constant sample paths of the fast process, the averaged drift is estimated via exact time averaging:
        \begin{align*}
        \tilde{b}(\hat{Z}_{n\Delta_1})=\tilde{b}_{R}(\hat{Z}_{n\Delta_1}):=& \frac{1}{R}\int_0^R b\left(\hat{Z}_{n\Delta_1}, \Lambda_{s}^{\hat{Z}_{n\Delta_1},i_0}\right)\dif s,
    \end{align*}
    where  $R$ denotes  the length of the time window for continuous-time averaging.
        This method can handle the infinite-state case and remove all discretization error from the micro time step; see Section \ref{SEC:PROPOSED} for more details.
    \end{enumerate}
    \end{enumerate}
    \end{enumerate}
    
 As we can see from the  Figure \ref{fig:algo_structure}, all these three algorithms share the same basic workflow. We first average the original process \(X_t^{\varepsilon}\) to obtain the solution \(\bar{X}_t\) of the averaged equation, which does not depends on scale parameter $\varepsilon$, and then construct numerical approximations for \(\bar{X}_t\). Through solving the corresponding linear systems, we get the expression of $\mu^x$ in the first algorithm, thus the averaged drift  \(\bar{b}(x)\) is  explicit, however the  computational cost of this algorithm increasing when state number $N$ increasing. In comparison, both the second and third algorithms use a two-layer structure that includes macro solver and micro solver.  For these two schemes, the cost of microscopic simulations is determined by the mixing rate of the fast process and does not depend on the scale parameter \(\varepsilon\). As a result, the total computational cost of all three methods  remains bounded even when \(\varepsilon \to 0\). This feature gives the proposed HMM methods a clear computational advantage over direct simulation of the original slow-fast coupled system.

\begin{figure}[htbp]
\centering
\begin{tikzpicture}[
    block/.style={
        rectangle, draw, thick,
        text width=3.8cm, text centered,
        inner sep=0.5em, minimum height=1.1em,
        font=\small
    },
    smallblock/.style={
        rectangle, draw, thick,
    minimum width=2.6cm,  
    minimum height=2.8em,  
    text width=2.0cm,      
    align=center,          
    inner sep=0pt,         
    font=\small
    },
    oval/.style={
        ellipse, draw, thick,
        text width=2.2cm, text centered,
        inner sep=0.4em, font=\small
    },
    arrow/.style={-Stealth, thick},
    dashconn/.style={dashed, thick},
    node distance=0.7cm and 0.7cm
]

\node (orig) [block] {Original process \(X_t^\varepsilon\)};
\node (avg) [block, below=of orig] {Averaged equation $\bar{X}_t$};
\node (em) [block, below=of avg] {EM scheme to $\bar{X}_t$};

\node (solv) [block, below =0.7cm of em, xshift=-3.5cm] {Solvable \(\mu^x\)};
\node (approx) [block,below =0.7cm  of em, xshift=3.5cm] {Approximable $\mu^x$};

\node (solve_lin) [smallblock, below=of solv] {Solve linear systems};

\node (macro)[smallblock, below=0.7cm of approx, xshift=-2.8cm]  {Macro\\ solver};
\node (micro) [smallblock, below=0.7cm of approx, xshift=2.8cm]  {Micro\\ solver};

\node (dtmc) [smallblock, below=0.7cm of micro, xshift=-1.5cm] {Construction of DTMC};
\node (ctmc) [smallblock, below=0.7cm of micro, xshift=1.5cm] {Construction of CTMC};

\node (algo1) [oval, below=3.3cm of solve_lin] {Algorithm 1};
\node (algo2) [oval, below=3.3cm of {$(macro.south)!0.01!(dtmc.south)$}] {Algorithm 2};
\node (algo3) [oval, below=2.2cm of {$(macro.south)!0.6!(ctmc.south)$}] {Algorithm 3};

\draw[arrow] (orig) -- (avg);
\draw[arrow] (avg) -- (em);
\draw[arrow] (em) -- (solv);
\draw[arrow] (em) -- (approx);
\draw[arrow] (solv) -- (solve_lin);
\draw[arrow] (approx) -- (macro);
\draw[arrow] (approx) -- (micro);
\draw[arrow] (micro) -- (dtmc);
\draw[arrow] (micro) -- (ctmc);

\draw[dashconn] (solve_lin) -- (algo1);
\draw[dashconn] (macro.south) -- (algo2.north);
\draw[dashconn] (dtmc.south) -- (algo2.north);
\draw[dashconn] (macro.south) -- (algo3.north);
\draw[dashconn] (ctmc.south) -- (algo3.north);

\end{tikzpicture}
\caption{Flowchart of three algorithms}
\label{fig:algo_structure}
\end{figure}
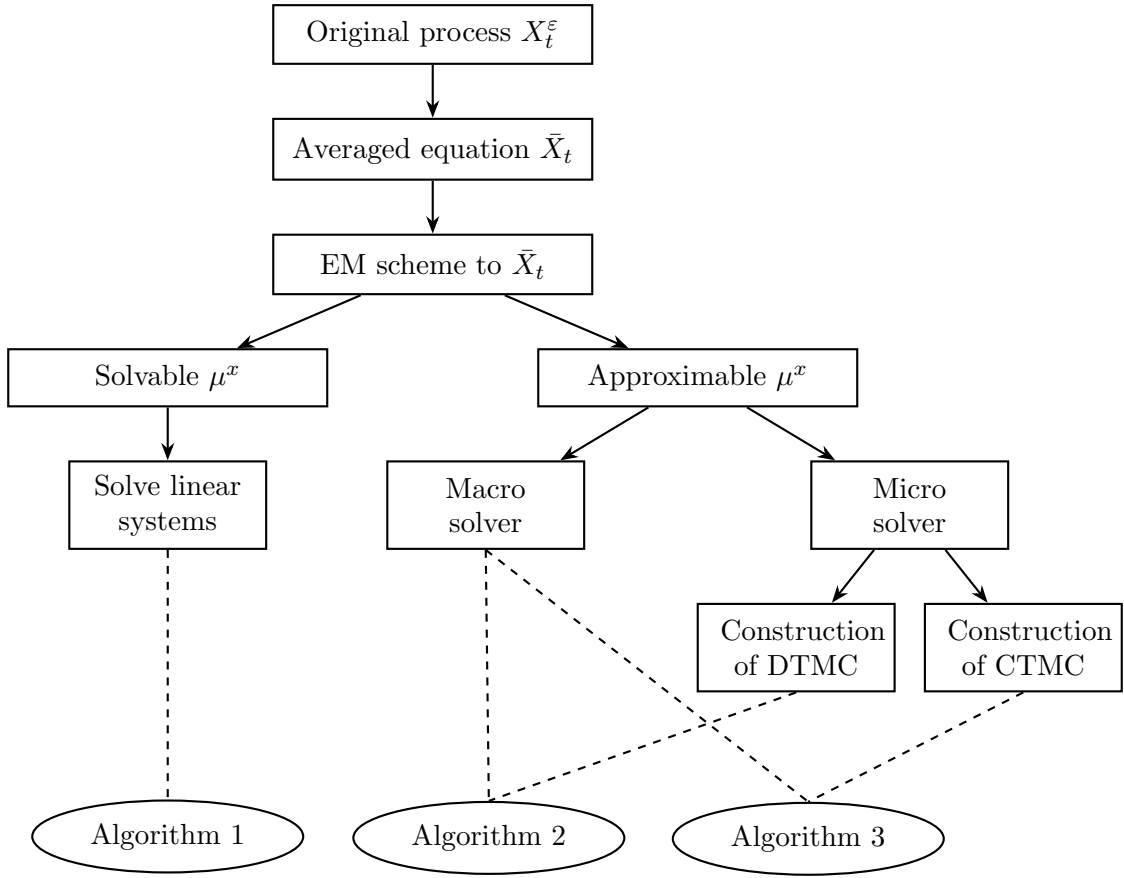

The rest of the paper is organized as follows. Section \ref{SEC:PRE} reviews preliminaries including standing assumptions and the averaging principle. Section \ref{SEC:Direct-EM} proposes Algorithm \ref{alg:hmm-euler} by solving the  invariant measure, and analyzes its convergence and efficiency.
Section \ref{SEC:CLASSIC} introduces Algorithm \ref{alg:discrete}, presents its convergence analysis,  and analyzes its limitations. Section \ref{SEC:PROPOSED} presents Algorithm \ref{alg:continuous} and its convergence analysis. Section \ref{SEC:Conclusion} presents the conclusion, outlook of this paper, and  several directions for future research.

\section{Preliminaries}\label{SEC:PRE}
Let $|\cdot|$ and $\|\cdot\|$ be  the standard Euclidean vector norm and matrix norm, respectively. Specifically, for $x = (x_k)_{1\leq k\leq n} \in \mathbb{R}^n$ and $\sigma = (\sigma_{kl})_{1\leq k\leq n,1\leq l\leq d} \in \mathbb{R}^n \otimes \mathbb{R}^d$,
\[
|x| := \left(\sum_{k=1}^n |x_k|^2\right)^{1/2},\quad \|\sigma\| := \left(\sum_{k=1}^n\sum_{l=1}^d |\sigma_{kl}|^2\right)^{1/2}.
\]
For $l_1,l_2 \in \mathbb{N}_+$, let $\mathscr{B}_b(\mathbb{S},\mathbb{R}^{l_1}\otimes\mathbb{R}^{l_2})$ be the space of all map $f(i): \mathbb{S} \to \mathbb{R}^{l_1}\otimes\mathbb{R}^{l_2}$ satisfying $\|f\|_\infty := \sup_{i\in\mathbb{S}} \|f(i)\| < \infty$. The total variation distance between probability measures $\mu$ and $\nu$ on $\mathbb{S}$ is denoted by $\|\mu-\nu\|_{\rm{var}}$. For a matrix $M = (m_{ij})_{i,j\in\mathbb{S}}$, denote $\|M\|_\ell := \sup_{i\in\mathbb{S}}\sum_{j\in\mathbb{S}} |m_{ij}|$.

Throughout this paper, we assume $b, \sigma$, and $Q$ in system \eqref{MULTI:EQ:0402:01}  satisfy the following conditions:
\begin{enumerate}[label=(H\arabic*)]
    \item\label{H1} There exists $C>0$ such that for $x, y \in \mathbb{R}^{n}, i, j \in \mathbb{S}$,
\begin{align*}
    &|b(x, i)-b(y, j)|\leq C(|x-y|+\mathbbm{1}_{\{i \neq j\}}), \quad 
\|\sigma(x)-\sigma(y)\| \leq C|x-y|.
\end{align*}
\item  (i) Assume $Q(x)=\left(q_{i j}(x)\right)_{i, j \in \mathbb{S}}: \mathbb{R}^{n} \rightarrow \mathbb{R}^{N} \otimes \mathbb{R}^{N}$ is measurable and conservative, i.e., 
\[
q_{i j}(x) \geq 0\text{~~for~any~~} i \neq j \in \mathbb{S},x \in \mathbb{R}^{n},  \quad \sum_{j \in \mathbb{S}} q_{i j}(x)=0\text{~for~any~} i \in \mathbb{S},x \in \mathbb{R}^{n}.
\]
(ii) Assume $Q(x)$ is irreducible, that is, for any $x \in \mathbb{R}^{n}$, the equations
\[
\mu^{x} Q(x)={0},    ~with~   \sum_{i \in \mathbb{S}} \mu_{i}^{x}=1
\]
have a unique solution $\mu^{x}=\left(\mu_{1}^{x}, \mu_{2}^{x}, \ldots, \mu_{N}^{x}\right)$ with $\mu_{i}^{x}>0$ for all $i \in \mathbb{S}$.\\
(iii) Let $P_{t}^{x}=\mathrm{e}^{Q(x) t}:=\left(p_{i j}^{x}(t)\right)_{i, j \in \mathbb{S}}$ be the transition probability matrix associated with $Q(x)$. $P_{t}^{x}$ is exponentially ergodic uniformly in $x$, i.e., there exist $C>0, \lambda>0$ such that
\[
\sup _{i \in \mathbb{S}, x \in \mathbb{R}^{n}}\left\|p_{i\cdot}^{x}(t)-\mu^{x}\right\|_{\mathrm{var}} \leq C \mathrm{e}^{-\lambda t}, \quad  \forall t>0.
\]
\item \label{ASSUMPTION:MULTI:01}   Assume there exists $C>0$ such that
\begin{equation}\label{EQ:LIPSCHITZ:Q:01}
    \|Q(x)-Q(y)\|_{\ell} \leq C|x-y|,\quad   \forall x, y \in \mathbb{R}^{n}.
\end{equation}
\[
K(x):=\sum_{i\in \mathbb{S}} \sum_{j \in \mathbb{S} \backslash\{i\}} q_{i j}(x) \leq C\left(1+|x|\right), \quad  \forall x \in \mathbb{R}^{n}.
\]
\end{enumerate}

\begin{remark}
Since the Lipschitz continuity of $b$ with respect to $x$ is assumed in (H1), the standard EM scheme can be applied to the averaged equation. This condition can be extended to more general settings, such as local Lipschitz continuity \cite[(2.1)]{SUNandXIE2025EJP}, in which case the standard EM scheme is replaced by the truncated EM (see, e.g., \cite{Mao2016}).
\end{remark}

\begin{remark}\label{MULTI:RMK:0416:01}
Under (H1), the Lipschitz continuity of the averaged drift \(\bar{b}\) follows by a similar argument to that in \cite[Lemma 4.2]{SUNandXIE2025EJP}, i.e.,
\[
|\bar{b}\left(x_{1}\right)-\bar{b}\left(x_{2}\right)| \leq C\left|x_{1}-x_{2}\right|,
\]
which implies the linear growth condition of $\bar{b}$, that is,
\[
|\bar{b}(x)| \leq C\left(1+|x|\right).
\]
\end{remark}

To proceed, we first recall the existence and uniqueness of the solution to  SDE \eqref{eq:averaged-equation}.
\begin{lemma}{$($\cite[Lemma 4.2]{SUNandXIE2025EJP}$)$}
	Suppose that conditions (H1)-(H3) hold. Then the averaged equation \eqref{eq:averaged-equation} admits a unique solution $(\bar{X}_{t})_{t \geq 0}$. Moreover, for any $T>0$ and $p>0$, there exists $C_{p, T}>0$ such that
	\[
	\mathbb{E}\left(\sup _{0 \leq t \leq T}|\bar{X}_{t}|^{p}\right) \leq C_{p, T}\left(1+|x|^{p}\right).
	\]
\end{lemma}

The averaging principle reduces the complexity of the slow-fast system, which becomes the baseline to develop the HMM. To facilitate this, we first recall the strong averaging principle with optimal convergence rate.
\begin{lemma}{$($\cite[Theorem 2.3]{SUNandXIE2025EJP}$)$}\label{lem:first1}
	Suppose that (H1)-(H3) hold. Then for $x_0 \in \mathbb{R}^{n}, i_0 \in \mathbb{S}, T>0$, and $p>0$, there exist constants $C_{p, T}>0, k_{p}>0$ such that
	\[
	\mathbb{E}\left(\sup_{t\in [0,T]}|X^{\varepsilon}_t-\bar{X}_t|^{p}\right) \leq C_{p, T}\left(1+|x|^{k_{p}}\right) \varepsilon^{p / 2},  \quad \forall \varepsilon \in(0,1],
	\]
	where $\bar{X}_{t}$ is the unique solution of the averaged equation \eqref{eq:averaged-equation}.
\end{lemma}

\section{Algorithm 1}\label{SEC:Direct-EM}
In this section, we present the first numerical scheme. Rather than directly applying the EM scheme to the slow process \(X_t^\varepsilon\), we only need to derive the EM scheme to the solution \(\bar{X}_t\) of the corresponding averaged equation, whose averaged coefficient $\bar{b}$ has a explicit expression through solving a linear system, that is, we construct an EM scheme \(Z_t\) for \(\bar{X}_t\), then together with Lemma \ref{lem:first1}, finally establish the error bound between \(X_t^\varepsilon\) and  \(Z_t\). Note that the number of states in the switching process is assumed to be finite throughout this section, thereby ensuring the solvability of the linear system associated with the invariant measure. Let $Z_0=x_0$. Algorithm \ref{alg:hmm-euler} is stated as follows: 
\begin{algorithm}[htbp]
\caption{}
\label{alg:hmm-euler}
\begin{algorithmic}[1]
\Require initial value $x_0$, step size $\Delta_1$, terminal time $T$
\State $Z_0 \gets x_0$
\For{$n = 0,1,2,\dots,\lfloor T/\Delta_1\rfloor-1$}
    \State $x \gets Z_{n\Delta_1}$
    \State Solve the linear system $\mu^x Q(x)=0$ subject to $\mu^x\mathbbm{1}=1$
    \State Compute the averaged drift $\bar{b}(x)=\sum_{i=1}^{N}b(x,i)\mu_i^x$
    \State Draw Brownian increment $\Delta W_n = W_{(n+1)\Delta_1}-W_{n\Delta_1}$
    \State $Z_{(n+1)\Delta_1} \gets Z_{n\Delta_1}+\bar{b}(Z_{n\Delta_1})\Delta_1+\sigma(Z_{n\Delta_1})\Delta W_n$
\EndFor
\State \Return $\{{Z}_{n\Delta_1}\},~0\leq n\leq \lfloor T/\Delta_1\rfloor$
\end{algorithmic}
\end{algorithm}

    For a given step size $\Delta_1 \in (0,1)$, define $t(\Delta_1):=\lfloor t/\Delta_1\rfloor\Delta_1$. We recall the  standard EM scheme of the averaged equation \eqref{eq:averaged-equation}:
	\begin{equation*}
		Z_{(n+1)\Delta_1}  =   Z_{n\Delta_1} +\bar{b}(Z_{n\Delta_1}) \Delta_1+ \sigma(Z_{n\Delta_1})\Delta W_n.
	\end{equation*}
    With a slight abuse of notation, the continuous-time interpolated version is denoted by \(Z_t\):
\begin{equation*}
	\dif Z_t = \bar{b}(Z_{t(\Delta_1)})\dif t+ \sigma (Z_{t(\Delta_1)})\dif W_t.
\end{equation*}

\subsection{Strong convergence of Algorithm 1}
Now, we first establish the   strong convergence between  $\bar{X}_t$ and $Z_t$, which is the standard result (see, e.g., \cite[Theorem 2.7.3]{MAOXUERONG2008BOOK}). For the convenience of the reader, we provide a brief proof. 
\begin{lemma}\label{lem:second2}
	Suppose that (H1)-(H3) hold. Then, for any $T>0$, $x_0\in\mathbb{R}^d$, and $p\geq 2$, there exists a constant $C_{x_0,T,p}$ such that for any $\Delta_1\in (0,1],$
	\begin{equation*}
		\mathbb{E}\left(\sup _{0 \leq t \leq T}\left|\bar{X}_{t}-Z_{t}\right|^{p}\right) \leq C_{x_0,T,p} \Delta_1^{p/2}.
	\end{equation*}
\end{lemma}
\begin{proof}
    By definition of $\bar{X}_{t}$ and $Z_t$, one has
	\begin{align*}
		\bar{X}_{t}-Z_t = \int_{0}^{t}\left(\bar{b}(\bar{X}_{s})-\bar{b}(Z_{s(\Delta_1)})\right)  \dif s+\int_{0}^{t}\left(\sigma(\bar{X}_{s})-\sigma(Z_{s(\Delta_1)})\right)  \dif W_s.
	\end{align*}
	It follows from the Burkholder-Davis-Gundy inequality, H\"older's inequality, and the $C_r$-inequality that for any $t\leq T$,
	\begin{align*}
		\mathbb{E}\left(\sup_{0\leq s\leq t}|\bar{X}_s-{Z}_s|^p\right)
		&\leq C_{p,T}\int_{0}^{t}\E|\bar{b}(\bar{X}_s)-\bar{b}(Z_{s(\Delta_1)})|^p\dif s\\
		&\quad + C_p\E\left(\int_{0}^{t}\|\sigma(\bar{X}_s)-\sigma(Z_{s(\Delta_1)})\|^2\dif s\right)^{p/2}\\
		&\leq C_{p,T}\int_{0}^{t}\E|\bar{X}_s-Z_{s}|^p\dif s+C_{p,T}\int_{0}^{t}\E|Z_s-Z_{s(\Delta_1)}|^p\dif s,
	\end{align*}
	where we used the Lipschitz continuity of $\bar{b}$ and $\sigma$ in the last step. By the standard moment estimate for the EM scheme (see, e.g., \cite[Lemma 2.6.2]{MAOXUERONG2008BOOK}), one has $\E|Z_s-Z_{s(\Delta_1)}|^p\leq C_{p,T}\Delta_1^{p/2}$. Consequently,
	\begin{align*}
        \mathbb{E}\left(\sup_{0\leq s\leq t}|\bar{X}_s-{Z}_s|^p\right)
        &\leq C_{x_0,T,p}\Delta_1^{p/2}+C_{p,T}\int_{0}^{t}\E\left(\sup_{0\leq u\leq s}|\bar{X}_u-Z_{u}|^p\right)\dif s.
	\end{align*}
    The Gr\"onwall inequality implies the desired result.
\end{proof}

Combining Lemmas \ref{lem:first1} and \ref{lem:second2}, we directly arrive at the following result.
\begin{theorem}\label{Thm:qi3}
    Suppose that (H1)-(H3) hold. Then, for any $T>0$, $x_0\in\mathbb{R}^d$, $i_0\in\mathbb{S}$, and $p\geq 2$, there exists a constant $C_{x_0,T,p}$ such that for any $\varepsilon\in(0,1]$ and $\Delta_1\in (0,1]$,
	\begin{align*}
		\mathbb{E}\left(\sup_{0\leq t\leq T}|X^{\varepsilon}_t-Z_t|^p\right) \leq C_{x_0,T,p} \left(\varepsilon^{p/2}+\Delta_1^{p/2}\right).
	\end{align*}
\end{theorem}
\begin{remark}
    Since the invariant measure \(\mu^x\) of the frozen fast  process is exactly computable, the corresponding averaged drift \(\bar{b}(x)\) is directly available in numerical simulations without further approximation. Therefore, we may directly employ \(Z_t\) as an estimator of \(\bar{X}_t\), which further yields strong convergence between \(X_t^\varepsilon\) and \(Z_t\).
\end{remark}

In what follows, we present two examples to illustrate the proposed method.

\subsection{Numerical experiments of Algorithm 1}
\noindent\textbf{Example (a)} We implement the above approach to carry out numerical simulations for the example  presented in Introduction (see \eqref{Eq:fcz}). In this example, the unique invariant probability measure admits a simple closed form $\mu=(\mu_1,\mu_2)=(2/3,1/3)$. Moreover,  the corresponding averaged equation is
\[
\d \bar{X}_t = \left(-\frac{4}{3} \bar{X}_t + 1\right) \d t + \frac{1}{5} \d W_t, \quad \bar{X}_0 = 1.
\]
This system corresponds to a standard Ornstein–Uhlenbeck process admitting an explicit solution:
\[
\bar{X}_t = \mathrm{e}^{-{4}t/{3} } + \frac{3}{4}\left(1 - \mathrm{e}^{-{4}t/{3}}\right) + \frac{1}{5} \int_0^t  \mathrm{e}^{-{4}(t-s)/3} \d W_s.
\]
\begin{figure}[htbp]
	\centering
	\includegraphics[width=1\linewidth]{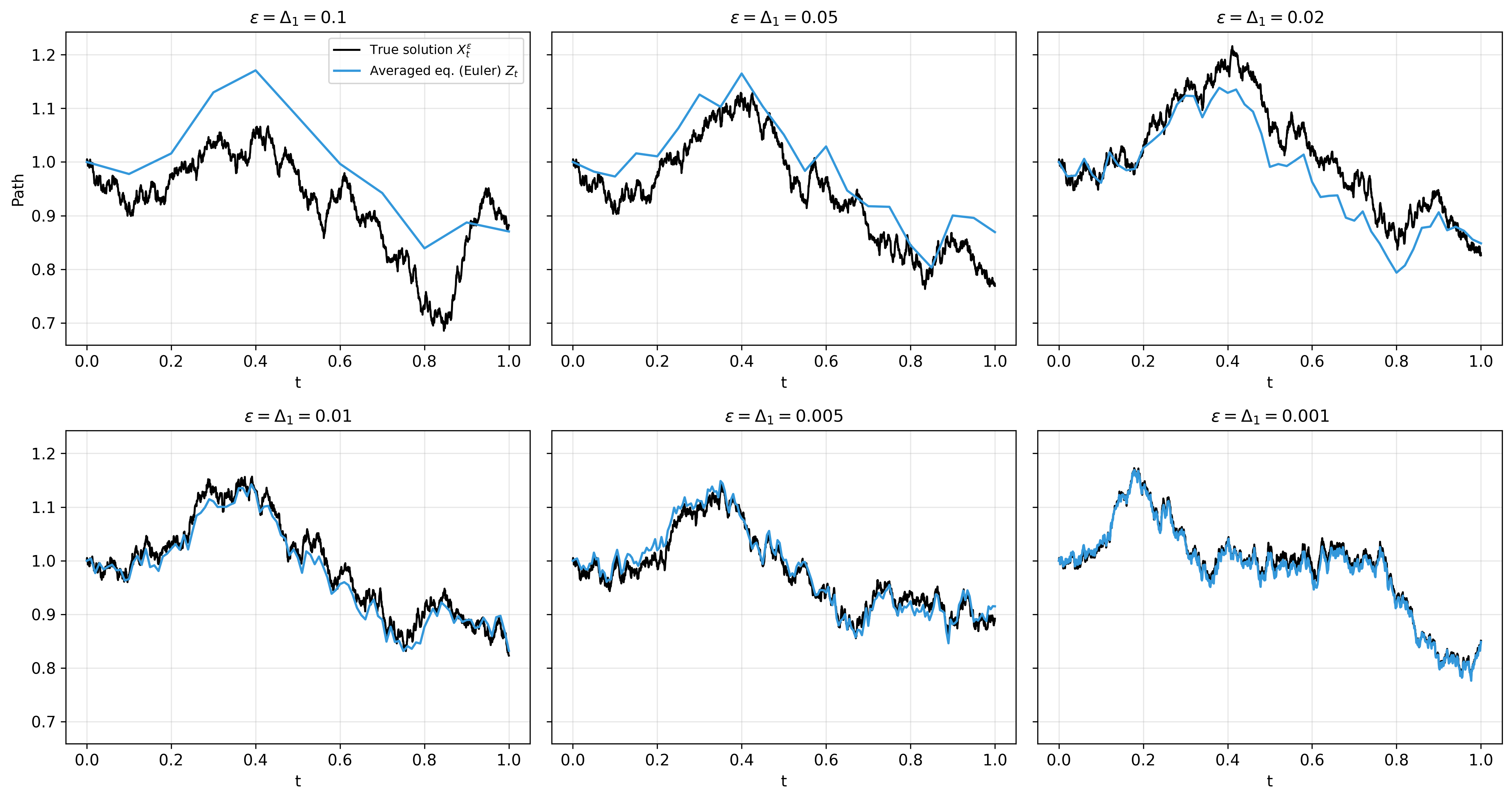}
	\caption{Comparison of sample paths of \(X_t^\varepsilon\) and \(Z_t\) as \(\varepsilon=\Delta_1 \downarrow 0\)}
	\label{fig:Averaging_principle01}
\end{figure}
We randomly select one sample path, set the terminal time \(T=1\), and impose \(\varepsilon=\Delta_1\) with \(\varepsilon=\Delta_1 = 0.1,\,0.05,\,0.02,\,0.01,\,0.005,\,0.002,\,0.001\). The corresponding approximation results are shown in Figure \ref{fig:Averaging_principle01}. From Figure \ref{fig:Averaging_principle01}, we observe that noticeable deviations between sample trajectories appear when \(\varepsilon=\Delta_1\) takes large values. As \(\varepsilon\) and \(\Delta_1\) decrease simultaneously,  deviations between sample trajectories diminish markedly, and \(Z_t\) provides a close approximation to \(X_t^\varepsilon\).  Compared with the results in Figure \ref{fig:direct_euler01}, this demonstrates that the EM approximation based on the averaged equation can overcome the divergence issue arising from direct EM discretization of \(X_t^\varepsilon\) for small \(\varepsilon\).
\begin{figure}[htbp]
	\centering
	\includegraphics[width=0.6\linewidth]{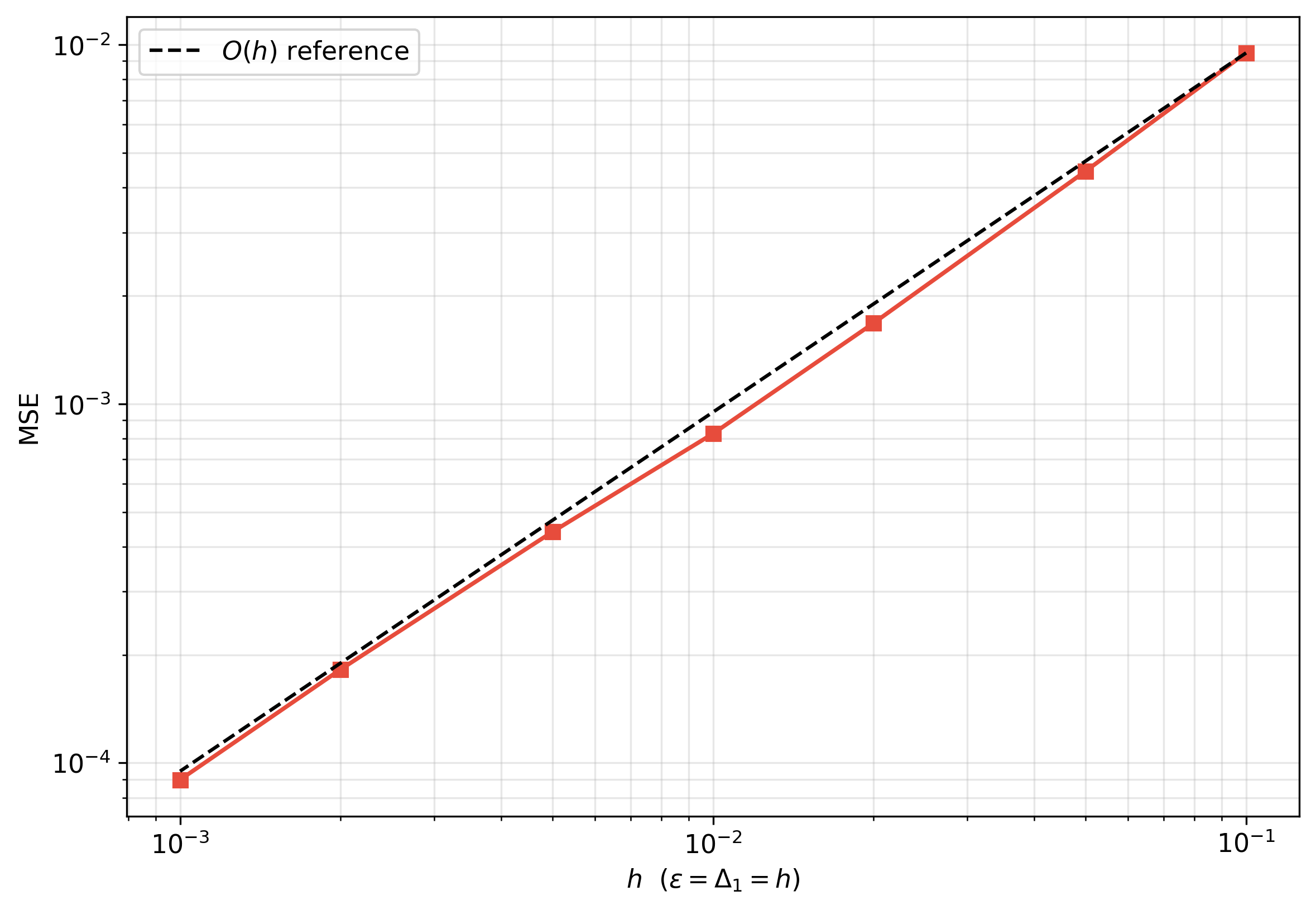}
	\caption{Log-log plot of the MSE between \(X_{0.5}^{\varepsilon}\) and \(Z_{0.5}\) against \(h=\varepsilon=\Delta_1\).}
	\label{fig:Convergence_of_Direct-EM01}
\end{figure}

To further verify the theoretical result stated in Theorem \ref{Thm:qi3}, we compute the sample mean squared error (MSE) using 500 sample trajectories:
\begin{equation}
    \mathbb{E}|X_{t^{*}}^{\varepsilon}-Z_{t^{*}}|^2 \approx \frac{1}{500}\sum_{j=1}^{500}\big|X^{\varepsilon,(j)}_{t^{*}}-Z_{t^{*}}^{(j)}\big|^2,
\end{equation}
where $X^{\varepsilon,(j)}_{t^{*}}$ and $Z_{t^{*}}^{(j)}$, $1\leq j\leq 500,$ are sequences of independent copies of $X_{t^{*}}^{\varepsilon}$ and $Z_{t^{*}}$, respectively. Note that $X^{\varepsilon,(j)}_{t^{*}}$ and $Z_{t^{*}}^{(j)}$ are generated by a same Brownian motion.  Figure \ref{fig:Convergence_of_Direct-EM01} illustrates the MSE between  \(X_t^{\varepsilon}\) and  \(Z_t\) evaluated at the fixed time \(t^{*}=0.5\), for seven distinct step sizes \(h=\varepsilon=\Delta_1=0.1,\, 0.05,\, 0.02,\, 0.01,\, 0.005,\, 0.002,\, 0.001\). The data points roughly lie along a straight line on the log–log scale, demonstrating that the MSE decays at first order with respect to \(h\). This numerically confirms the first-order convergence, in the mean-square sense, of the EM scheme associated with the averaged equation.

\begin{remark}
In numerical convergence analysis, the MSE typically obeys a power-law scaling with respect to the step size $h$, i.e., $\mathrm{MSE} \propto h^\alpha$. Taking the logarithm on both sides reduces the relation to a linear form:
\begin{equation*}
\log(\mathrm{MSE}) = \alpha \log(h) + C
\end{equation*}
where $C$ is a constant independent of $h$. On a log-log coordinate system, this relation appears as a straight line, whose slope $\alpha$ corresponds exactly to the order of convergence. This graphical representation offers an intuitive means to verify the error decay rate.
In this experiment, the MSE is given as the following table. 
\begin{table}[htbp]
  \centering
  \rmfamily\upshape 
  \begin{tabular}{rccccccc}
  \toprule
  $h_i~(1\leq i\leq 7)$      & 0.1      & 0.05     & 0.02     & 0.01     & 0.005    & 0.002    & 0.001    \\
  \midrule
  $MSE_i~(1\leq i\leq 7)$    & 0.009483 & 0.004445 & 0.001677 & 0.000826 & 0.000440 & 0.000182 & 0.000090 \\
  \bottomrule
  \end{tabular}
\end{table}
\\
As a result, we can see that $\alpha=\frac{\log (MSE_{i}/MSE_{i-1})}{\log (h_{i}/h_{i-1})}\approx 1$ for all $2\leq i\leq 7$.
\end{remark}

\noindent\textbf{Example (b)} 
As a running example throughout the paper, we consider a slow-fast system whose fast component is a high-dimensional switching process with a ring structure, inspired by the random walk on a ring lattice studied in \cite[Section 3.8]{cocconi2020entropy}.
This example is suitable for verifying our method because the invariant measure of the ring-shaped CTMC has no explicit expression for general $N$ ($N<\infty$).
We consider the following slow-fast system:
\begin{equation}\label{eq:ring-example-system}
	\left\{\begin{array}{l}
		\dif X_{t}^{\varepsilon}=b\left(X_{t}^{\varepsilon}, \Lambda_{t}^{\varepsilon}\right) \dif t+ \sigma\left(X_{t}^{\varepsilon}\right)\dif W_{t}, \\
		\mathbb{P}\left(\Lambda_{t+\Delta}^{\varepsilon}=j \mid \Lambda_{t}^{\varepsilon}=i, X_{s}^{\varepsilon}, \Lambda_{s}^{\varepsilon}, s \leq t\right)=\left\{\begin{array}{l}
			\varepsilon^{-1} q_{i j}\left(X_{t}^{\varepsilon}\right) \Delta+o(\Delta), \quad i \neq j, \\
			1+\varepsilon^{-1} q_{i i}\left(X_{t}^{\varepsilon}\right) \Delta+o(\Delta), \quad i=j,
		\end{array}\right. \\
		\left(X_{0}^{\varepsilon}, \Lambda_{0}^{\varepsilon}\right)= (7/10, 1) \in \mathbb{R} \times \mathbb{S},~ t\geq 0,
	\end{array}\right.
\end{equation}
where the coefficients are defined  by
\[
b(x, i) = x \cdot \sin\left(\frac{2\pi i}{N}\right),\quad \sigma(x) = \frac{4}{5}x,
\]
and the state-dependent transition rate matrix $Q(x) = (q_{ij}(x))_{i,j \in \mathbb{S}}$ is defined as follows:
\begin{align*}
	q_{i(i+1)}(x)& = 2 + \sin\left(\frac{2\pi i}{N} + x\right), \quad\forall i=1,...,N-1,\;q_{N1}(x)=2+\sin x;\\
	q_{i(i-1)}(x) &= 2 + \cos\left(\frac{2\pi i}{N} + x\right), \quad\forall i=2,...,N,\;q_{1N}(x)=2+\cos x;\\
q_{11}(x) &= -\left(q_{12}(x) + q_{1N}(x)\right),\quad q_{NN}(x) = -\left(q_{N1}(x) + q_{N(N-1)}(x)\right);\\
	q_{ii}(x) &= -\left(q_{i(i+1)}(x) + q_{i(i-1)}(x)\right),\quad \forall i=2,...,N-1;\\
	q_{ij}(x) &= 0, \quad \forall 1<|i-j|< N-1.
\end{align*}
For a clearer illustration of the generator, we present the state transition diagram of a ring-shaped CTMC for the case \(N = 7\) as follows.
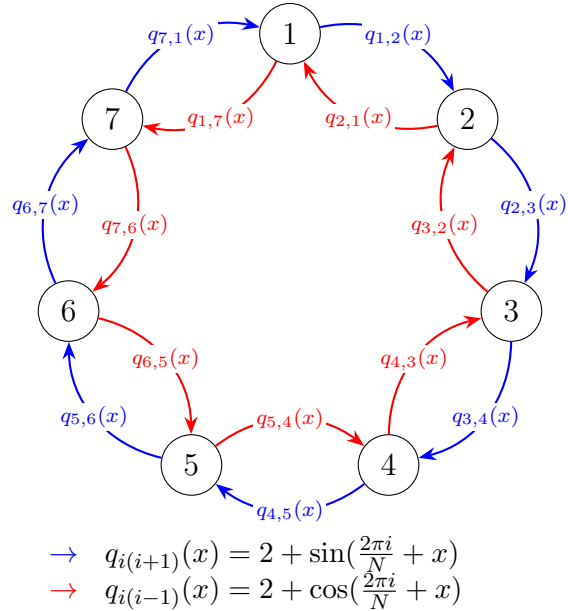
\begin{figure}[htbp]
	\centering
	\begin{tikzpicture}[
		>=Stealth,
		state/.style={circle,draw,minimum size=8mm,inner sep=0pt},
		cw/.style={blue,thick},
		ccw/.style={red,thick},
		]
		
		\def\N{7}
		\def\R{3}
		
		\foreach \i in {1,...,\N} {
			\node[state] (n\i) at ({90-(\i-1)*360/\N}:\R) {\i};
		}
		
		\foreach \i in {1,...,\N} {
			\pgfmathtruncatemacro{\j}{mod(\i,\N)+1}
			\draw[cw,->] (n\i) to[bend left=38] 
			node[midway,fill=white,inner sep=1pt,font=\tiny] {$q_{\i,\j}(x)$} (n\j);
		}
		
		\foreach \i in {1,...,\N} {
			\pgfmathtruncatemacro{\k}{int(mod(\i-2+\N,\N)+1)}
			\draw[ccw,->] (n\i) to[bend left=38] 
			node[midway,fill=white,inner sep=1pt,font=\tiny] {$q_{\i,\k}(x)$} (n\k);
		}
		
		\node[anchor=north west,font=\small] at (-3.5,-3.5) {
			\begin{tabular}{ll}
				\textcolor{blue}{$\rightarrow$} & $q_{i(i+1)}(x)=2+\sin(\frac{2\pi i}{N}+x)$ \\
				\textcolor{red}{$\rightarrow$} & $q_{i(i-1)}(x)=2+\cos(\frac{2\pi i}{N}+x)$
			\end{tabular}
		};
	\end{tikzpicture} 
	\caption{State transition diagram of a ring-shaped CTMC ($N=7$)}
	\label{fig:ctmc-ring}
\end{figure}

By the averaging principle, as $\varepsilon \to 0$, the slow process $X_t^\varepsilon$ converges strongly to the solution $\bar{X}_t$ of the averaged equation:
\[
\dif \bar{X}_t = \bar{b}(\bar{X}_t) \dif t + \sigma(\bar{X}_t)\dif W_t, \quad
	\bar{X}_0 = 7/10,
\]
where the averaged drift coefficient is:
\[
\bar{b}(x) =  \sum_{i=1}^Nx  \sin\left(\frac{2\pi i}{N}\right)\mu_i^x,
\]
and $\mu^x = (\mu_1^x, \dots, \mu_N^x)^{\top}$ is the unique invariant measure of the frozen CTMC with generator $Q(x)$.


For this example, the true solution \(X_t^{\varepsilon}\) is unavailable. To further verify the theoretical result stated in Theorem \ref{Thm:qi3}, we thus evaluate the sample MSE between  the averaged equation \(\bar{X}_t\) and its EM scheme \(Z_t\) using 500 sample trajectories for the case of $N=100$:
\begin{equation}
    \mathbb{E}|\bar{X}_{t^{*}}-Z_{t^{*}}|^2 \approx \frac{1}{500}\sum_{j=1}^{500}\big|\bar{X}^{(j)}_{t^{*}}-Z_{t^{*}}^{(j)}\big|^2,
\end{equation}
where $\bar{X}^{(j)}_{t^{*}}$ and $Z_{t^{*}}^{(j)}$, $1\leq j\leq 500,$ are sequences of independent copies of $\bar{X}^{(j)}_{t^{*}}$ and $Z_{t^{*}}$, respectively. We take the EM numerical solution of the averaged equation $\bar{X}_{t}$ with an extremely fine step size $10^{-4}$ as the reference solution $\bar{X}_{t}^{ref}$.  Figure \ref{fig:direct_method_efficiency_N} similarly demonstrates that the MSE decays at first order with respect to \(\Delta_1\), which confirms the theoretical result established in Theorem \ref{Thm:qi3}.
\begin{figure}[htbp]
	\centering
	\includegraphics[width=0.6\linewidth]{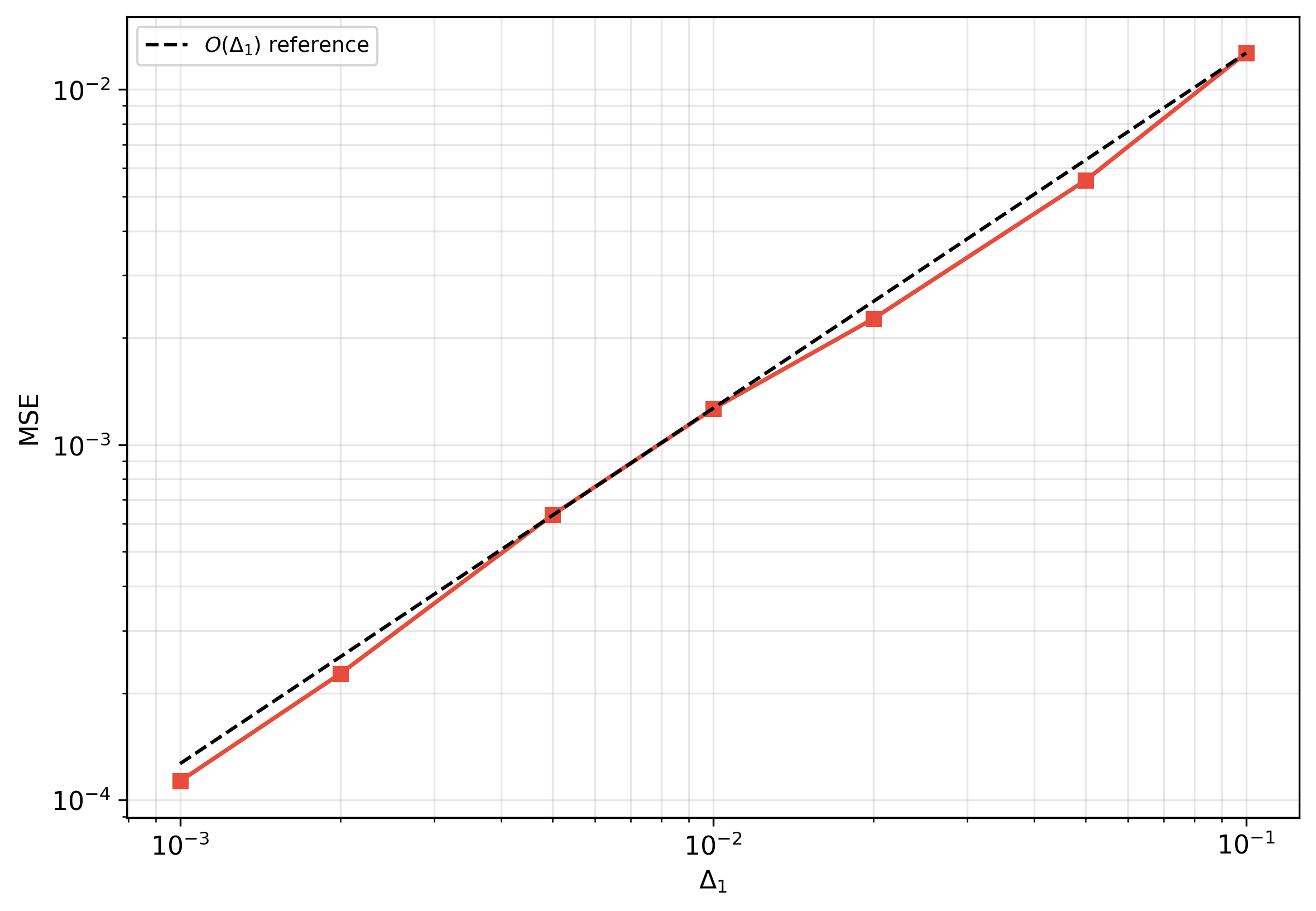}
	\caption{Log-log plot of the MSE between \(\bar{X}_{0.5}^{ref}\) and \(Z_{0.5}\) against \(\Delta_1\).}
	\label{fig:direct_method_efficiency_N}
\end{figure}

Finally, we set the target MSE at \(t^{*}=0.5\) to be \(10^{-4}\) and report the computational efficiency  for this example.  As shown in Table \ref{table:ring efficient}, the runtime of Algorithm \ref{alg:hmm-euler} rises rapidly as $N$ increases. This indicates that the computational cost grows significantly with the size of the Markov state space, which reveals the limitation of Algorithm \ref{alg:hmm-euler} when facing Markov chains with a large number of states.
\begin{table}[htbp]
  \centering
  \caption{Runtime (s) of Algorithm 1  with different $N$.}
  \begin{tabular}{rc}
    \toprule
    $N$ & Algorithm 1  \\
    \midrule
   10   & 0.0011 \\
   50   & 0.0194 \\
   100  & 0.0394 \\
   200  & 0.1439 \\
   500  & 1.2223 \\
   1000 & 5.4191 \\
    \bottomrule
  \end{tabular}
  \label{table:ring efficient}
\end{table}

    For this ring-shaped CTMC, the invariant measure $\mu^x$ has no explicit analytical expression for general $N$ and $x$. Direct computation of \(\mu^x\) requires solving the $N$-dimensional linear system \(\mu^x Q(x) = 0\) and \(\mu^x \mathbbm{1} = 1\). Typically, such linear systems are solved by the singular value decomposition with a computational complexity of \(O(N^3)\)(see \cite[Chapter 8.7]{golub2013matrix} for example), which becomes prohibitive for large $N$.
\section{Algorithm 2}\label{SEC:CLASSIC}

From Example (b) in Section \ref{SEC:Direct-EM}, Algorithm \ref{alg:hmm-euler} exhibits \(O(N^3)\) computational complexity, which makes it computationally infeasible for large $N$. Moreover, the invariant measure of the fast process generally does not admit a closed-form expression. Consequently, the averaged drift \(\bar{b}(x)\) associated with  \(Z_t\)  is unavailable analytically.

Recall that for a frozen CTMC $\Lambda_t^{x}$ with the invariant measure $\mu^x$, if follows from the  classic ergodicity property that for a proper function $f:\mathbb{S}\to \mathbb{R}$,
\begin{equation}\label{EQ:FROZEN:ERGO}
\frac{1}{M}\sum_{m=0}^{M-1} f(\Lambda_{m\Delta}^x)\to	\mu^x(f),\quad M\to\infty. 
\end{equation}
Note that the left-hand side of the above equation depends only on the values of the Markov chain at discrete time points. 
Motivated by this observation,
we can approximate \(\bar{b}(x)\) by its DTMC simulation-based estimator \(\tilde{b}_{\Delta2,M}(x)\), which further yields the improved numerical approximation \(\hat{Z}_t\) for \(Z_t\). Consequently, the problem reduces to constructing a DTMC with a prescribed transition probability matrix. Note that the number of states in the switching process is also assumed to be finite throughout this section, thereby ensuring the solvability of the transition probability matrix $\e^{Q(x)t}$ for any $t>0$.

Then,
we present Algorithm \ref{alg:discrete} below:
\begin{algorithm}[htbp]
\caption{}
\label{alg:discrete}
\begin{algorithmic}[1]
\Require initial value $(x_0,i_0)$, macro step $\Delta_1$, micro step $\Delta_2$, number of micro steps $M$, terminal time $T$, generator $Q(\cdot)$
\State Initialize $\hat{Z}_0\gets x_0$.
\For{$n = 0,1,2,\dots,\lfloor T/\Delta_1\rfloor$-1}  \Comment{\textbf{Macro solver}}
    \State Freeze the slow variable $x \gets \hat{Z}_{n\Delta_1}$
    \State Set transition matrix $P(x) = \mathrm{e}^{\Delta_2 Q(x)}$
    \State Initialize $\Lambda_0^{x,i_0} \gets i_0$, $\tilde{b}_{\Delta2,M}(x) \gets 0$
    \For{$m = 0,1,\dots,M-1$} \Comment{\textbf{Micro solver}}
        \State Generate $\Lambda_{(m+1)\Delta_2}^{x,i_0}$ from $\Lambda_{m\Delta_2}^{x,i_0}$ via the inverse transform method with $P(x)$
        \State $\tilde{b}_{\Delta2,M}(x) \gets \tilde{b}_{\Delta2,M}(x) + b\!\left(x,\Lambda_{m\Delta_2}^{x,i_0}\right)$
    \EndFor
    \State $\tilde{b}_{\Delta2,M}(x) \gets \dfrac{1}{M}\,\tilde{b}_{\Delta2,M}(x)$ 
    \State Draw Brownian increment $\Delta W_n = W_{(n+1)\Delta_1}-W_{n\Delta_1}$
    \State $\hat{Z}_{(n+1)\Delta_1} \gets \hat{Z}_{n\Delta_1}
        + \tilde{b}_{\Delta2,M}(\hat{Z}_{n\Delta_1})\Delta_1
        + \sigma(\hat{Z}_{n\Delta_1})\Delta W_n$ \Comment{\textbf{EM update}}
\EndFor
 \State \Return $\{\hat{Z}_{n\Delta_1}\},~0\leq n\leq \lfloor T/\Delta_1\rfloor$
\end{algorithmic}
\end{algorithm}

\subsection{Strong convergence of Algorithm 2}
In this section, we establish the strong convergence result for Algorithm \ref{alg:discrete}; see Theorem \ref{Thm:strongr}. First, we establish the mean squared error between the average drift $\bar{b}(x)$ and the discrete-time drift estimator $\tilde{b}_{\Delta2,M}(x)$ via the Poisson equation, which is inspired from \cite{MST2010}; see Lemma  \ref{lem:guji}. Next, we derive the strong convergence of both the standard EM scheme \(Z_t\) and the multiscale EM numerical solution \(\hat{Z}_t\); see Lemma \ref{lem:third3}. Finally, by combining the above results with the averaging principle, we obtain the desired strong convergence conclusion. For readability, we postpone the proofs of Lemmas \ref{lem:guji}, \ref{lem:juguj}, and \ref{lem:third3} to the end of this section.

Recall the Poisson equation associated with the generator $Q(x)$. Let $\Lambda_t^{x,i}$ be the CTMC with generator $Q(x)$ and initial state $i$. Suppose that $F(x,i) = (F^1(x,i), \dots, F^n(x,i))^T$ with $F^l(x,\cdot) \in \mathbb{R}^{N}$, $l = 1,2,\dots,n$, satisfies the centering condition:
\begin{equation}\label{eq:central}
	\sum_{i \in \mathbb{S}} F^l(x,i)\mu_i^x = 0, \quad \forall x \in \mathbb{R}^n, \, l = 1,2,\dots,n,
\end{equation}
and $Q(x) = (q_{ij}(x))_{i,j \in \mathbb{S}}$. Considering the following Poisson equation on $\mathbb{S}$:
\begin{equation}\label{eq:poisson}
	-Q(x)\Phi(x,\cdot)(i) = F(x,i),
\end{equation}
which is equivalent to
\[
-Q(x)\Phi^l(x,\cdot)(i) = -\sum_{j \in \mathbb{S}} q_{ij}(x)\Phi^l(x,j) = F^l(x,i), \quad l = 1,2,\dots,n,
\]
where $\Phi(x,\cdot) = (\Phi^1(x,\cdot), \dots, \Phi^n(x,\cdot))$ with $\Phi^l(x,\cdot) \in \mathbb{R}^{N}$, $l = 1,2,\dots,n$. For Poisson equation \eqref{eq:poisson}, we have the following result:
\begin{theorem}{$($\cite[Theorem 2.2]{SUNandXIE2025EJP}$)$}\label{thm:poisson1}
Suppose that (H2) holds, $F$ satisfies the centering condition  \eqref{eq:central} with $\|F(x,\cdot)\|_{\infty} < \infty$. Define
\[
\Phi(x,i) = \int_0^{\infty} \mathbb{E}F(x,\Lambda_t^{x,i}) \dif t.
\]
Then $\Phi(x,i)$ solves \eqref{eq:poisson} and satisfies
\begin{equation}\label{eq:yizhi}
    \|\Phi(x,\cdot)\|_\infty \leq C \|F(x,\cdot)\|_\infty.
\end{equation}
 Moreover, if $Q\in C^1(\mathbb{R}^n;\mathbb{R}^{N}\otimes\mathbb{R}^{N})$ and $F\in C^1(\mathbb{R}^n\times\mathbb{S};\mathbb{R}^n)$, then there exists a constant $C>0$ such that for any $x\in\mathbb{R}^n$,
\[
\|\partial_x\Phi(x,\cdot)\|_\infty \leq C\left(\|F(x,\cdot)\|_\infty\|\nabla Q(x)\|_\ell + \|\partial_x F(x,\cdot)\|_\infty\right).
\]
\end{theorem}

Using the Poisson equation technique, we obtain the following error bound between $\bar{b}(x)$ and $\tilde{b}_{\Delta_2,M}(x)$.
\begin{lemma}\label{lem:guji}
Suppose that (H1)-(H3) hold. Then, for any $p\geq 2$, there exists a constant $C_p>0$ such that
    \begin{equation*}
        \E\left|\bar{b}(x)-\tilde{b}_{\Delta_2,M}(x)\right|^p\leq C_p(1+|x|)^{3p/2}\left(\Delta_2+\frac{1}{(M\Delta_2)^{p/2}}\right).
    \end{equation*}
\end{lemma}
\begin{proof}
We first consider the case $M\Delta_2<1$. Since both $\bar{b}(x)$ and $b(x,i)$ are bounded by $C(1+|x|)$ under (H1), we have
\begin{equation*}
\left|\bar{b}(x)-\tilde{b}_{\Delta_2,M}(x)\right|
\leq |\bar{b}(x)|+\frac{1}{M}\sum_{m=0}^{M-1}\left|b(x,\widetilde{\Lambda}_{m\Delta_2}^{x,i_0})\right|
\leq 2C(1+|x|),
\end{equation*}
and consequently
\begin{equation*}
\mathbb{E}\left|\bar{b}(x)-\tilde{b}_{\Delta_2,M}(x)\right|^p\leq C_p(1+|x|)^p.
\end{equation*}
As $M\Delta_2<1$ and $p\geq 2$ implies $(M\Delta_2)^{-p/2}>1$ and $p\leq 3p/2$, it follows that
\begin{equation*}
\mathbb{E}\left|\bar{b}(x)-\tilde{b}_{\Delta_2,M}(x)\right|^p
\leq C_p(1+|x|)^{3p/2}(M\Delta_2)^{-p/2}
\leq C_p(1+|x|)^{3p/2}\left(\Delta_2+\frac{1}{(M\Delta_2)^{p/2}}\right).
\end{equation*}
Thus, the desired estimate holds trivially when $M\Delta_2<1$. 

In what follows, we assume without loss of generality that $M\Delta_2\geq 1$. Since the error $\E|\bar{b}(x)-\tilde{b}_{\Delta_2,M}(x)|^p$ depends only on the finite-dimensional distribution of the discrete chain $\{\Lambda_{m\Delta_2}^{x,i_0}\}_{m\geq 0}$, we can perform the estimation on an equivalent probability space without loss of generality. To this aim, we first construct a homogeneous CTMC with the generator $Q(x)$ via Skorokhod's representation.
For any $x\in \mathbb{R}^d$, let
\[
\Delta_{1 2}(x)=\left[0, q_{1 2}(x)\right), \Delta_{1 l}(x)=\left[\sum_{j=2}^{l-1}q_{1j}(x),\sum_{j=2}^{l}q_{1j}(x)\right),\quad l\geq 3,
\]
and for each $k \geq 2$ and $x\in \mathbb{R}^d$, let
\[
\Delta_{k 1}(x)=\left[0, q_{k 1}(x)\right), \Delta_{k l}(x)=\left[\sum_{j=1,j\neq k}^{l-1}q_{kj}(x),\sum_{j=1,j\neq k}^{l}q_{kj}(x)\right),\quad l>1,l\neq k.
\]
Note that for each $k \in \mathbb{S} $ and $x \in \mathbb{R}^{d}$, $\{\Delta_{kl}(x):l\in \mathbb{S} \}$ are disjoint intervals, and the length of $\Delta_{kl}(x)$ equals $q_{kl}(x)$, which is bounded above by $C(1+|x|)$ thanks to (H3). We then define the jump function $h: \mathbb{R}^{d} \times \mathbb{S} \times [0, +\infty) \to \mathbb{R}$ by
\[
h(x, k, u)=\sum_{l \in \mathbb{S}}(l-k) \mathbbm{1}_{\triangle_{k l}(x)}(u).
\]
That is, for each $k \in \mathbb{S}$, if $u \in \triangle_{k l}(x)$, then $h(x, k, u)=l-k$; otherwise $h(x, k, u)=0$.

Let $N(\d t, \d u)$ be a Poisson random measure 
with Lebesgue measure on $[0, +\infty)$ as its characteristic measure. The evolution of $\widetilde{\Lambda}_t^{x,i_0}$ is given by
\begin{equation}\label{eq:skorohod}
	\text{d}\widetilde{\Lambda}_t^{x,i_0}=\int_{\left[0, H\right]} h(x, \widetilde{\Lambda}_{t-}^{x,i_0}, u) N(\text{d} t, \text{d} u), \quad \widetilde{\Lambda}_0^{x,i_0}=i_0.
\end{equation}
It is standard that $\widetilde{\Lambda}_t^{x,i_0}$ is a homogeneous CTMC with generator $Q(x)$.
By the basic property of CTMC, its sampling at discrete times $\{\widetilde{\Lambda}_{m\Delta_2}^{x,i_0}\}_{m\geq 0}$ is a DTMC with transition matrix $\mathrm{e}^{Q(x)\Delta_2}$, which shares the same finite-dimensional distribution as the original discrete chain $\{\Lambda_{m\Delta_2}^{x,i_0}\}_{m\geq 0}$.
It follows from the $C_r$-inequality that for any $p\geq 2$,
\begin{align}\label{eq:buli}
	     \nonumber\E|\bar{b}(x)-\tilde{b}_{\Delta_2,M}(x)|^p &= \E\left(\left|\bar{b}(x)-\frac{1}{M}\sum_{m=0}^{M-1} b\left(x, \Lambda_{m\Delta_2}^{x,i_0}\right)\right|^p\right)\\
     \nonumber&= \E\left(\left|\bar{b}(x)-\frac{1}{M}\sum_{m=0}^{M-1} b\left(x, \widetilde{\Lambda}_{m\Delta_2}^{x,i_0}\right)\right|^p\right)\\
	     &\leq 2^{p-1}\E\left(\left|\bar{b}(x)-\frac{1}{M\Delta_2}\int_{0}^{M\Delta_2}b(x,\widetilde{\Lambda}_{s}^{x,i_0})\dif s\right|^p\right)\\
	    \nonumber &\quad+2^{p-1}\E\left(\left|\frac{1}{M\Delta_2}\int_{0}^{M\Delta_2}b(x,\widetilde{\Lambda}_{s}^{x,i_0})\dif s-\frac{1}{M}\sum_{m=0}^{M-1} b\left(x, \widetilde{\Lambda}_{m\Delta_2}^{x,i_0}\right)\right|^p\right)\\
	    \nonumber &\leq \frac{2^{p-1}}{(M\Delta_2)^p}\E\left(\left|\int_{0}^{M\Delta_2}\left(\bar{b}(x)-b(x,\widetilde{\Lambda}_s^{x,i_0})\right)\dif s\right|^p\right)\\
	 \nonumber &\quad+\frac{2^{p-1}}{M\Delta_2}\sum_{m=0}^{M-1}\E\int_{m\Delta_2}^{(m+1)\Delta_2}\left|b(x,\widetilde{\Lambda}_s^{x,i_0})- b(x, \widetilde{\Lambda}_{m\Delta_2}^{x,i_0})\right|^p\dif s.
\end{align}

We estimate the two terms separately in what follows.

\noindent\textbf{Ergodic error.} We estimate the ergodic error term via the Poisson equation approach.
Consider the Poisson equation associated with the generator $Q(x)$:
\begin{equation}\label{eq:poi}
	-Q(x) \Phi(x,\cdot)(i) =b(x,i)-\bar{b}(x).
\end{equation}
By Theorem \ref{thm:poisson1},  \eqref{eq:poi} admits a solution $\Phi(x,\cdot)$, and furthermore,
\begin{equation}\label{MULTI:POISSON:EQ:042i}
	\|\Phi(x,\cdot)\|_{\infty} \le C\|b(x,\cdot)-\bar{b}(x)\|_{\infty} \le C(1+|x|).
\end{equation}
Applying It\^o's formula for jump process $\widetilde{\Lambda}_{t}^{x,i_0}$ to $\Phi(x,\cdot)$, we obtain
\begin{align}\label{eq:ghshg}
	\Phi(x,\widetilde{\Lambda}_{t}^{x,i_0}) &=\Phi(x,i_0) +\int_{0}^{t}Q(x) \Phi(x,\cdot)(\widetilde{\Lambda}_{s}^{x,i_0})\dif s\\
	\nonumber&\quad + \int_{0}^{t}\int_{[0,\infty)}\left(\Phi(x,\widetilde{\Lambda}_{s-}^{x,i_0}+h(x,\widetilde{\Lambda}_{s-}^{x,i_0},z))-\Phi(x,\widetilde{\Lambda}_{s-}^{x,i_0})\right)\widetilde {N}(\dif s,\dif z).
\end{align}
Furthermore, by \eqref{MULTI:POISSON:EQ:042i}, \eqref{eq:ghshg}, and Kunita’s first inequality (see, e.g., \cite[Theorem 4.4.23]{applebaum2009levy}), we obtain that for $M\Delta_2\geq 1$,
\begin{align}\label{eq:daiwa}
	\nonumber&\quad\E\left(\left|\int_{0}^{M\Delta_2}\left(\bar{b}(x)-b(x,\widetilde{\Lambda}_{s}^{x,i_0})\right)\dif s\right|^p\right)\\
    \nonumber&=\mathbb{E}\left(\left|\int_{0}^{M\Delta_2}Q(x) \Phi(x,\cdot)(\widetilde{\Lambda}_{s}^{x,i_0})\dif s\right|^p\right)\\
	\nonumber&\leq 2^{p-1}\mathbb{E}[|\Phi(x,\widetilde{\Lambda}_{M\Delta_2}^{x,i_0}) -\Phi(x,i_0)|^p]\\
	&\quad+2^{p-1}\mathbb{E}\left(\left|\int_{0}^{M\Delta_2}\int_{[0,\infty)}\left(\Phi(x,\widetilde{\Lambda}_{s-}^{x,i_0}+h(x,\widetilde{\Lambda}_{s-}^{x,i_0},z))-\Phi(x,\widetilde{\Lambda}_{s-}^{x,i_0})\right)\widetilde N(\dif s,\dif z) \right|^p\right)\\
	\nonumber&\leq C_p(1+|x|)^{p}+C_p\mathbb{E}\left(\int_{0}^{M\Delta_2}\int_{[0,C(1+|x|)]}\left|\Phi(x,\widetilde{\Lambda}_{s-}^{x,i_0}+h(x,\widetilde{\Lambda}_{s-}^{x,i_0},z))-\Phi(x,\widetilde{\Lambda}_{s-}^{x,i_0})\right|^2\dif z\dif s \right)^{p/2}\\
    \nonumber&\quad +C_p\mathbb{E}\left(\int_{0}^{M\Delta_2}\int_{[0,C(1+|x|)]}\left|\Phi(x,\widetilde{\Lambda}_{s-}^{x,i_0}+h(x,\widetilde{\Lambda}_{s-}^{x,i_0},z))-\Phi(x,\widetilde{\Lambda}_{s-}^{x,i_0})\right|^p\dif z\dif s \right)\\
	\nonumber&\leq C_p(1+|x|)^{p}+C_p(1+|x|)^{3p/2}(M\Delta_2)^{p/2}+C_p(1+|x|)^{p+1}M\Delta_2\\
   \nonumber &\leq C_p(1+|x|)^{3p/2}(M\Delta_2)^{p/2}.
\end{align}

\noindent\textbf{Discretization error.}  We estimate the discretization error term on each interval $[m\Delta_2, (m+1)\Delta_2]$. 
Since the holding time of
$\widetilde{\Lambda}^{x,i_0}_t$  at state $i$ is exponentially distributed
with parameter $q_i(x)$, we have its transition probability $p^x_{ii}(t)\geq \mathrm{e}^{-q_i(x)t}$, and
consequently
\begin{equation}\label{eq:ctmc-hold}
\sum_{j\neq i}p^{x}_{ij}(t)=1-p^x_{ii}(t)
\leq 1-\mathrm{e}^{-q_i(x)t}\leq q_i(x)t,
\qquad \forall\, t\geq 0.
\end{equation}
Applying \eqref{eq:ctmc-hold}
together with (H3) yields
\begin{align}\label{eq:buneng}
		\nonumber &\quad \E\int_{m\Delta_2}^{(m+1)\Delta_2}\left|b(x,\widetilde{\Lambda}_{s}^{x,i_0})- b(x, \widetilde{\Lambda}_{m\Delta_2}^{x,i_0})\right|^p\dif s \\
	\nonumber	 &= \E\int_{m\Delta_2}^{(m+1)\Delta_2}\left|b(x,{\widetilde{\Lambda}}_{s}^{x,i_0})- b(x, \widetilde{\Lambda}_{m\Delta_2}^{x,i_0})\right|^p \mathbbm{1}_{\{\widetilde{\Lambda}_{s}^{x,i_0}\neq\widetilde{\Lambda}_{m\Delta_2}^{x,i_0} \}}\dif s  \\
	 	&=\E\sum_{i\in\mathbb{S}} \sum_{j\neq i}\int_{m\Delta_2}^{(m+1)\Delta_2}\left|b(x,j)- b(x, i)\right|^p \mathbbm{1}_{\{\widetilde{\Lambda}_{s}^{x,i_0}=j\}}\mathbbm{1}_{\{\widetilde{\Lambda}_{m\Delta_2}^{x,i_0} =i\}}\dif s \\
	\nonumber	&\leq C_p(1+|x|)^{p}\E\sum_{i\in\mathbb{S}} \sum_{j\neq i}\int_{m\Delta_2}^{(m+1)\Delta_2} \mathbbm{1}_{\{\widetilde{\Lambda}_{m\Delta_2}^{x,i_0} =i\}}\E\left[\mathbbm{1}_{\{\widetilde{\Lambda}_{s}^{x,i_0}=j\}}\mid\widetilde{\Lambda}_{m\Delta_2}^{x,i_0} =i \right]\dif s \\
\nonumber	 	&\leq C_p(1+|x|)^{p}\E\sum_{i\in\mathbb{S}} \int_{m\Delta_2}^{(m+1)\Delta_2} \mathbbm{1}_{\{\widetilde{\Lambda}_{m\Delta_2}^{x,i_0} =i\}}q_i(x)(s-m\Delta_2)\dif s \\
\nonumber	 	&\leq C_p(1+|x|)^{p+1} \int_{0}^{\Delta_2} s  \d s \leq C_p(1+|x|)^{p+1}\Delta_2^2.
	\end{align}
Inserting \eqref{eq:buneng} and \eqref{eq:daiwa} into \eqref{eq:buli} yields
\[
\E|\bar{b}(x)-\tilde{b}_{\Delta_2,M}(x)|^p\leq C_p(1+|x|)^{3p/2}\left(\Delta_2+\frac{1}{(M\Delta_2)^{p/2}}\right).
\]
The proof is complete.
\end{proof}

\begin{remark}
Since the switching process is a jump process, the final convergence rate $\Delta^2_2$ does not depend on the power $p\geq 2$ in the studying the \textbf{Discretization error} in the proof of Lemma \ref{lem:guji}.  Consequently, the final error in Lemma \ref{lem:guji} remains $\Delta_2$.
\end{remark}

 Due to computational convenience, we present the continuous-time interpolated version, and with a minor abuse of notation, continue to denote it by $\hat{Z}_t$:
	\begin{equation}\label{EQ:discrete-time multiscale:01}
		\begin{split}
			\dif \hat{Z}_t &=\tilde{b}_{\Delta_2,M}(\hat{Z}_{t(\Delta_1)})\dif t+\sigma (\hat{Z}_{t(\Delta_1)})\dif W_t\\
			&= \frac{1}{M}\sum_{m=0}^{M-1}b(\hat{Z}_{t(\Delta_1)},\Lambda_{m\Delta_2}^{\hat{Z}_{t(\Delta_1)},i_0})\dif t+\sigma (\hat{Z}_{t(\Delta_1)})\dif W_t,\quad \hat{Z}_0= x_0,
		\end{split}
	\end{equation}
To proceed, we have the following moment estimate for $\hat{Z}_t$ defined by \eqref{EQ:discrete-time multiscale:01}.
\begin{lemma}\label{lem:juguj}
	Suppose that (H1)-(H3) hold. Then, for any $x_0\in \mathbb{R}^n$, $i_0\in \mathbb{S}$, $T>0,$ $p\geq2$, and $M\geq 1,$ there exists a constant $C_{x_0,T,p}>0$ such that
	\begin{equation*}
		\mathbb{E}\left(\sup_{0\leq t\leq T}|\hat{Z}_t|^{p}\right)\leq C_{x_0,T,p}.
	\end{equation*}
\end{lemma}

\begin{proof}
	Using It\^o's formula for $|\hat{Z}_t|^p$ for any $p\geq 4$, one has for any  $0\leq t\leq T,$
	\begin{align*}
		|\hat{Z}_t|^{p}=&  \left|x_{0}\right|^{p}+p \int_{0}^{t}|\hat{Z}_s|^{p-2}\langle  \hat{Z}_s, \sigma(\hat{Z}_{s(\Delta_1)}) \mathrm{d} W_s\rangle\\
        &+\frac{p}{2} \int_{0}^{t}|\hat{Z}_s|^{p-2}\bigg(2\langle \hat{Z}_s, \tilde{b}_{\Delta_2,M}(\hat{Z}_{s(\Delta_1)})\rangle+\|\sigma(\hat{Z}_{s(\Delta_1)})\|^{2}\\
	&\qquad\qquad\qquad\qquad\qquad\qquad+(p-2)|\hat{Z}_s|^{-2}|\hat{Z}_s\cdot \sigma(\hat{Z}_{s(\Delta_1)})|^2\bigg) \mathrm{d} s.
	\end{align*}
	By the Burkholder-Davis-Gundy inequality, Young's inequality, and the linear growth of $\sigma$, one has
	\begin{align*}
		\mathbb{E}\left(\sup_{0\leq s\leq t}|\hat{Z}_s|^{p}]\right)&\leq |x_0|^{p}+\frac{1}{2}\mathbb{E}\left(\sup_{0\leq s\leq t}|\hat{Z}_s|^{p}\right) +C_{T,p}\int_{0}^{t}\mathbb{E}|\hat{Z}_s|^{p}\dif s+C_{T,p}\int_{0}^{t}\mathbb{E}|\hat{Z}_{s(\Delta_1)}|^{p}\dif s\\
		&\quad+p\mathbb{E}\left(\sup_{0\leq s\leq t}\int_{0}^{s}|\hat{Z}_r|^{p-2}\langle \hat{Z}_r, \tilde{b}_{\Delta_2,M}(\hat{Z}_{r(\Delta_1)})\rangle \mathrm{d} r\right).
	\end{align*}
	By Young's inequality, Jensen's inequality, the definition of $\tilde{b}_{\Delta_2,M}$,  and $\|b(x,\cdot)\|_{\infty}\leq C(1+|x|)$, we obtain
	\begin{align*}
		&\quad \mathbb{E}\left(\sup_{0\leq s\leq t}\int_{0}^{s}|\hat{Z}_r|^{p-2}\langle \hat{Z}_r, \tilde{b}_{\Delta_2,M}(\hat{Z}_{r(\Delta_1)})\rangle \mathrm{d} r\right) \\
		&\leq C_p\mathbb{E}\int_{0}^{t}|\hat{Z}_s|^{p}\mathrm{d} s+C_p\mathbb{E}\int_{0}^{t} |\tilde{b}_{\Delta_2,M}(\hat{Z}_{s(\Delta_1)})|^p \mathrm{d} s\\
		&=C_p\mathbb{E}\int_{0}^{t}|\hat{Z}_s|^{p}\mathrm{d} s+\frac{C_p}{(M\Delta_2)^p}\mathbb{E}\int_{0}^{t} \left|\int_0^{M\Delta_2}{b}(\hat{Z}_{s(\Delta_1)},\Lambda_{u(\Delta_2)}^{x_0,i_0})\d u\right|^p \mathrm{d} s\\
		&\leq C_p\mathbb{E}\int_{0}^{t}|\hat{Z}_s|^{p}\mathrm{d} s+\frac{C_p}{M\Delta_2}\mathbb{E}\int_{0}^{t} \int_0^{M\Delta_2}|{b}(\hat{Z}_{s(\Delta_1)},\Lambda_{u(\Delta_2)}^{x_0,i_0})|^p\d u \mathrm{d} s\\
		&\leq C_{T,p}+C_{p}\mathbb{E}\int_{0}^{t}|\hat{Z}_s|^{p}\dif s+ C_{p}\mathbb{E}\int_{0}^{t}|\hat{Z}_{s(\Delta_1)}|^{p}\dif s\\
		&\leq C_{T,p}+C_{p}\int_{0}^{t}\mathbb{E}\left(\sup_{0\leq r\leq s}|\hat{Z}_r|^{p}\right)\dif s.
	\end{align*}
	Combining the above results and using Gr\"onwall's inequality, we can  derive the result.
    \end{proof}

Then we turn to prove the strong convergence of the standard EM scheme $Z_t$ and the multiscale numerical solution $\hat{Z}_t$.
\begin{lemma}\label{lem:third3}
	Suppose that (H1)-(H3) hold. Then, for any $T>0$, $i_0\in\mathbb{S}$, $x_0\in\mathbb{R}^d$, and $p\geq 2$, there exists a constant $C_{x_0,T,p}$ such that for any $\Delta_1\in [0,1]$,
	\begin{align*}
		\mathbb{E}\left(\sup_{0\leq t\leq T}|Z_t-\hat{Z}_t|^p\right) \leq C_{x_0,T,p}\left(\Delta_2+\frac{1}{(M\Delta_2)^{p/2}}\right).
	\end{align*}
\end{lemma}
\begin{proof}
	By definition of $Z_t$  and $\hat{Z}_t$, one has
	\begin{align*}
		Z_t-\hat{Z}_t = \int_{0}^{t}\left(\bar{b}(Z_{s(\Delta_1)})-\tilde{b}_{\Delta_{2},M}(\hat{Z}_{s(\Delta_1)})\right)  \dif s+\int_{0}^{t}\left(\sigma(Z_{s(\Delta_1)})-\sigma(\hat{Z}_{s(\Delta_1)})\right)  \dif W_s.
	\end{align*}
	It follows from the Burkholder-Davis-Gundy inequality, H\"older's inequality, and the $C_r$-inequality that for any $t\in [0,T]$, we have
	\begin{align*}
		\mathbb{E}\left(\sup_{0\leq s\leq t}|Z_s-\hat{Z}_s|^p\right) &\leq C_{p,T} \int_{0}^{t}\E|\bar{b}(Z_{s(\Delta_1)})-\tilde{b}_{\Delta_{2},M}(\hat{Z}_{s(\Delta_1)})|^p\dif s\\
		&\quad + C_p\E\left(\int_{0}^{t}\|\sigma(Z_{s(\Delta_1)})-\sigma(\hat{Z}_{s(\Delta_1)})\|^2\dif s\right)^{p/2}\\
		&\leq
		C_{p,T} \int_{0}^{t}\E|\bar{b}(Z_{s(\Delta_1)})-\bar{b}(\hat{Z}_{s(\Delta_1)})|^p\dif s\\
		&\quad+C_{p,T} \int_{0}^{t}\E|\bar{b}(\hat{Z}_{s(\Delta_1)})-\tilde{b}_{\Delta_{2},M}(\hat{Z}_{s(\Delta_1)})|^p\dif s\\
		&\quad + C_{p,T}\int_{0}^{t}\E\|\sigma(Z_{s(\Delta_1)})-\sigma(\hat{Z}_{s(\Delta_1)})\|^p\dif s=:\sum_{i=1}^{3}\mathcal{S}_i.
	\end{align*}
	For $\mathcal{S}_1$ and $\mathcal{S}_3$, it follows from the Lipschitz continuity of  $\bar{b}$ and $\sigma$ that
	\begin{align*}
		\mathcal{S}_1+\mathcal{S}_3\leq C_{p,T}\int_{0}^{t} \E|Z_{s(\Delta_1)}-\hat{Z}_{s(\Delta_1)}|^p\dif s.
	\end{align*}
	For $\mathcal{S}_2,$ by Lemmas \ref{lem:guji} and \ref{lem:juguj},
	\begin{align*}
		&\quad \E|\bar{b}(\hat{Z}_{s(\Delta_1)})-\tilde{b}_{\Delta_2,M}(\hat{Z}_{s(\Delta_1)})|^p \\
        &= \E\left|\frac{1}{M\Delta_2}\int_{0}^{M\Delta_2}\left(\bar{b}(\hat{Z}_{s(\Delta_1)})-b(\hat{Z}_{s(\Delta_1)},\Lambda_{u(\Delta_2)}^{\hat{Z}_{s(\Delta_1)},i_0})\right)\dif u\right|^p \\
		& =\E\left[\E\left(\left|\frac{1}{M\Delta_2}\int_{0}^{M\Delta_2}\left(\bar{b}(\hat{Z}_{s(\Delta_1)})-b(\hat{Z}_{s(\Delta_1)},\Lambda_{u(\Delta_2)}^{\hat{Z}_{s(\Delta_1)},i_0})\right)\dif u\right|^p\Bigg|\hat{Z}_{s(\Delta_1)}\right)\right]\\
		&= \E\left[\E\left(\left|\bar{b}(x)-\tilde{b}_{\Delta2,M}(x)\right|^p\right)\bigg|_{x=\hat{Z}_{s(\Delta_1)}}\right]\\
        &\leq C_{p}\E\left[(1+|\hat{Z}_{s(\Delta_1)}|^{3p/2})\left(\Delta_2+\frac{1}{(M\Delta_2)^{p/2}}\right)\right]\\
        &\leq C_{x_0,T,p}\left(\Delta_2+\frac{1}{(M\Delta_2)^{p/2}}\right).
	\end{align*}
	Combining the above estimates and applying Gr\"onwall's inequality, we have
	\begin{align*}
		\mathbb{E}\left(\sup_{0\leq t\leq T}|Z_t-\hat{Z}_t|^p\right)\leq C_{x_0,T,p}\left(\Delta_2+\frac{1}{(M\Delta_2)^{p/2}}\right).
	\end{align*}
	The proof is complete.
\end{proof}

Based on Lemmas \ref{lem:first1}, \ref{lem:second2} and \ref{lem:third3}, we can easily obtain the strong convergence between the slow component $X^{\varepsilon}_t$ and the multiscale numerical solution $\hat{Z}_t$.

\begin{theorem}\label{Thm:strongr}
	Suppose that (H1)-(H3) hold. Then, for any $T>0$, $x_0\in\mathbb{R}^d$, $i_0\in\mathbb{S}$, and $p\geq 2$, there exists a constant $C_{x_0,T,p}$ such that for any $\varepsilon\in(0,1]$ and $\Delta_1\in[0,1]$,
	\begin{align*}
		\mathbb{E}\left(\sup_{0\leq t\leq T}|X^{\varepsilon}_t-\hat{Z}_t|^p\right)\leq C_{x_0,T,p} \left(\varepsilon^{p/2}+\Delta_1^{p/2}+\Delta_2+\frac{1}{(M\Delta_2)^{p/2}}\right).
	\end{align*}
\end{theorem}
\begin{remark}\label{rem:error}
    From Theorem \ref{Thm:strongr}, we observe that discrete errors $1/(M\Delta_2)$ appear in the error terms between \(X^{\varepsilon}_t\) and \(\hat{Z}_t\). This arises because we approximate \(\mu^x(f)\) in \eqref{EQ:FROZEN:ERGO} using a DTMC $\{\Lambda_{m\Delta_2}^{x}\}_{0\leq m\leq M}$, rather than the CTMC $\Lambda_t^{x}$ itself.
\end{remark}

\subsection{Numerical experiments of Algorithm 2}
We carry out the same MSE computation for \textbf{Example (b)}  as before ($N=100$).
Figure \ref{fig:dtmc_multiscale_efficiency_N} demonstrates that the MSE decays at the expected
algebraic rate as \(h=\Delta_1=\Delta_2=(M\Delta_2)^{-1}=0.1,\, 0.05,\, 0.02,\, 0.01,\, 0.005,\, 0.002\) tends to zero, which confirms
the theoretical convergence result established for Algorithm \ref{alg:discrete} in Theorem \ref{Thm:strongr}.
\begin{figure}[htbp]
	\centering
	\includegraphics[width=0.6\linewidth]{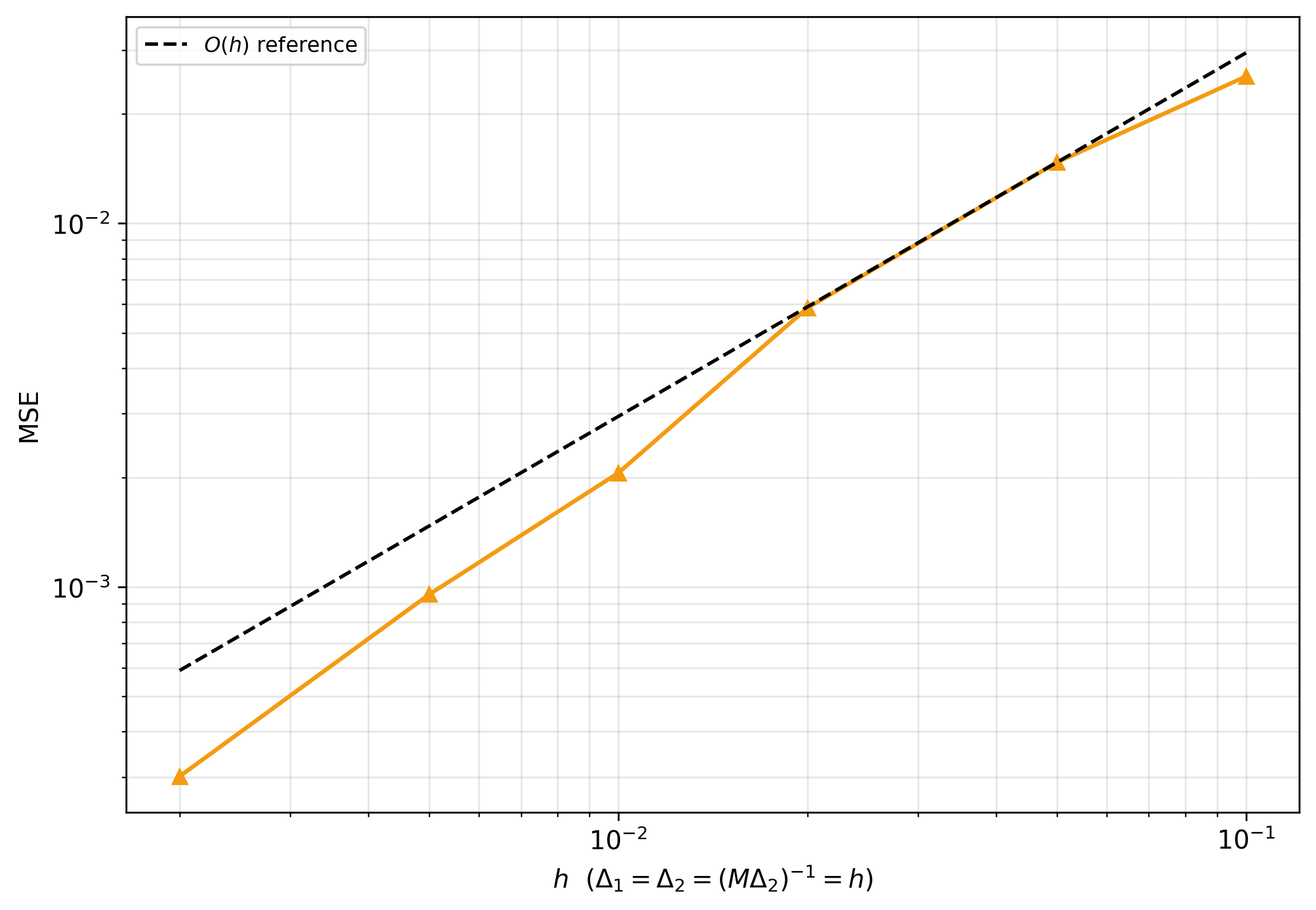}
	\caption{Log-log plot of the MSE between \(\bar{X}_{0.5}^{ref}\) and \(\hat{Z}_{0.5}\) against \(\Delta_1=\Delta_2=(M\Delta_2)^{-1}\).}
	\label{fig:dtmc_multiscale_efficiency_N}
\end{figure}

We set the target MSE at \(t^{*}=0.5\) to be \(10^{-4}\) and report the computational
efficiency for this example. As shown in Table \ref{tab:efficiency:01}, Algorithm \ref{alg:discrete}
is slower than Algorithm \ref{alg:hmm-euler}  when the state space is small, but its runtime grows much more
slowly as \(N\) increases, and it becomes faster than Algorithm \ref{alg:hmm-euler}  from \(N=200\) onwards,
with about five times speedup at \(N=1000\). This indicates that Algorithm \ref{alg:discrete} avoids the costly
computation of the invariant measure, which is the main bottleneck of Algorithm \ref{alg:hmm-euler}, and is
therefore more suitable for the fast process with a large number of states.

\begin{table}[htbp]
  \centering
  \caption{Runtime (s) comparison I with different $N$}
  \label{tab:efficiency:01}
  \begin{tabular}{rccc}
  \toprule
    $N$ & Algorithm 1 & Algorithm 2 & Speedup (2 vs 1)  \\
    \midrule
   10   & 0.0011 & 0.0039 & 0.28$\times$ \\
   50   & 0.0194 & 0.0289 & 0.67$\times$ \\
   100  & 0.0394 & 0.0517 & 0.76$\times$ \\
   200  & 0.1439 & 0.1480 & 0.97$\times$ \\
   500  & 1.2223 & 0.4343 & 2.81$\times$ \\
   1000 & 5.4191 & 1.0719 & 5.06$\times$ \\    
    \bottomrule
     \end{tabular}
\end{table}

\begin{remark}
Algorithm \ref{alg:discrete} does not solve the invariant measure \(\mu^x\) of the fast process.
At each macro step, it estimates the averaged drift by the sample mean of the drift
along a DTMC with \(M\) steps and step size \(\Delta_2\), where
\(\Delta_1=\Delta_2=(M\Delta_2)^{-1}=h\). Its per-iteration cost depends only on the
current state of the chain, and no \(N\)-dimensional linear system needs to be solved.
Consequently, for a fixed accuracy, the total cost of Algorithm \ref{alg:discrete} grows much more slowly
with \(N\) than that of Algorithm \ref{alg:hmm-euler}, although a sufficiently small \(\Delta_2\) (equivalently
a sufficiently large \(M\)) is needed to control the discretization and sampling errors
of the drift estimate.
\end{remark}


\section{Algorithm 3}\label{SEC:PROPOSED}
Note that Algorithms \ref{alg:hmm-euler} and \ref{alg:discrete} are both restricted to a finite state space in the switching process. To overcome this limitation, we introduce a third algorithm. Meanwhile, in order to eliminate the discretization error mentioned in Remark \ref{rem:error}, we adopt the original CTMC $\Lambda_{t}^x$ to approximate \(\mu^x(f)\), i.e.,  for any $f$,
\begin{equation*}
	\frac{1}{R}\int_{0}^{R} f(\Lambda_{s}^x)\dif s \to	\mu^x(f),\quad R\to\infty. 
\end{equation*}
Therefore, the key problem reduces to the numerical simulation of the CTMC $\Lambda_{t}^x$, for which we employ the Gillespie algorithm \cite{Gillespie1977JPC}. 
Now, let $(\alpha_t)_{t\ge 0}$ be a CTMC with generator $Q$. We construct its exact trajectory, which jointly determines the holding time and the jump target using two independent uniform random variables. Let $\alpha_0 = i_0$ be the deterministic initial state. The process is constructed recursively at jump times $T_0 < T_1 < T_2 < \cdots$:
\begin{algorithm}[H]
\caption*{\textbf{Gillespie algorithm}}
\label{alg:gillespie}
\begin{algorithmic}[1]
\Require generator $Q=(q_{ij})$,  initial state $\alpha_0=i_0$
\State Initialize $m\gets 0$, $T_0\gets 0$, $\alpha_{T_0}\gets i_0$
\For{$m = 0,1,2,\dots$}
    \State Let $k\gets\alpha_{T_m}$ be the current state
    \State Compute the total exit rate $q_k=-q_{kk}=\sum_{j\neq k}q_{kj}$
    \If{$q_k=0$ \textbf{(absorbing state)}}
        \State Set $T_{m+1}\gets\infty$, $\alpha_t\gets k$ for all $t\ge T_m$
        \State \textbf{terminate} the construction
    \Else
        \State Draw independent $U_{m},V_{m}\sim\mathrm{Uniform}(0,1)$
        \State  $\tau_{m+1}\gets-q_k^{-1}\log U_{m}$,
               and $T_{m+1}\gets T_m+\tau_{m+1}$
        \State Select the achieving state
        \[
        \ell^* = \min\left\{ l \in \{1,\dots,N\} : \sum_{j=1,j\neq k}^{l} q_{kj} \ge V_{m}q_k \right\}
        \]
        \State Update the state: $\alpha_{T_{m+1}}\gets\ell^*$
        \State For $t\in[T_m,T_{m+1})$, keep $\alpha_t\equiv\alpha_{T_m}=k$
    \EndIf
\EndFor
\end{algorithmic}
\end{algorithm}
We now present the proposed Algorithm \ref{alg:continuous}.

\begin{algorithm}[htbp]
\caption{}
\label{alg:continuous}
\begin{algorithmic}[1]
\Require initial value $(x_0,i_0)$, macro step $\Delta_1$, averaging window $R\ge1$, terminal time $T$, generator $Q(\cdot)$
\State Initialize $\hat{Z}_0\gets x_0$
\For{$n=0,1,2,\dots,\lfloor T/\Delta_1\rfloor-1$} \Comment{\textbf{Macro solver}}
    \State Freeze the slow variable $x\gets\hat{Z}_{n\Delta_1}$
    \State Initialize $T_{n,0}\gets0$, $\Lambda_{T_{n,0}}^{x,i_0}\gets i_0$, 
    \State Generate jump times $\{T_{n,m}\}_{m\geq 0}$ and simulate the frozen CTMC $(\Lambda_t^{x,i_0})_{t\ge0}$ on $[0,R]$ by the Gillespie algorithm
    \State $\eta_{n,R}\gets\inf\{k>0:T_{n,k}\ge R\}$
    \State Initialize $\tilde{b}_R(x)\gets0$
    \For{$m=0,1,\dots,\eta_{n,R}-1$}\Comment{\textbf{Micro solver} }
        \State $\tilde{b}_R(x)\gets\tilde{b}_R(x)+(T_{n,m+1}\wedge R-T_{n,m})\,b\!\left(x,\Lambda_{T_{n,m}}^{x,i_0}\right)$
    \EndFor
    \State $\tilde{b}_R(x)\gets\dfrac{1}{R}\,\tilde{b}_R(x)$, 
    \State Draw Brownian increment $\Delta W_n=W_{(n+1)\Delta_1}-W_{n\Delta_1}$
    \State $\hat{Z}_{(n+1)\Delta_1}\gets\hat{Z}_{n\Delta_1}
        +\bar{b}_R(\hat{Z}_{n\Delta_1})\Delta_1
        +\sigma(\hat{Z}_{n\Delta_1})\Delta W_n$ \Comment{\textbf{EM update}}
\EndFor
 \State \Return $\{\hat{Z}_{n\Delta_1}\},~0\leq n\leq \lfloor T/\Delta_1\rfloor$
\end{algorithmic}
\end{algorithm}
\begin{remark}
    If \(\mathbb{S}=\{1,2,\cdots\}\), then under the uniform boundedness condition
\(q_k \leq M, \; \forall k\in\mathbb{S},\)
which follows from  (H3), we can construct the CTMC using an argument similar to Gillespie’s method. This makes the method suitable for systems with a large or even infinite number of switching states, where direct matrix computation is infeasible.
    In what follows, we show that 
    \begin{align*}
    	\int_{0}^{R} b\left(\hat{Z}_{n\Delta_1}, \Lambda_{s}^{\hat{Z}_{n\Delta_1},i_0}\right) \dif s
    	= \sum_{m=0}^{\eta_{n,R}-1} \left(T_{n,m+1}\wedge R - T_{n,m}\right) b\left(\hat{Z}_{n\Delta_1}, \Lambda_{T_{n,m}}^{\hat{Z}_{n\Delta_1},i_0}\right),
    \end{align*}
    which follows directly from the piecewise constant property of the CTMC trajectory. The integral over $[0,R]$ reduces to a finite sum over holding intervals, which can be computed exactly with no discretization error. More precisely, assume that the process $(\Lambda_{t}^{\hat{Z}_{n\Delta_1},i_0})_{t\geq0}$ is a piecewise constant jump process with jump times $\{T_{n,m}\}_{m\geq0}$ satisfying
    \begin{equation*}
    	    0 = T_{n,0} < T_{n,1} < T_{n,2} < \cdots, \quad \lim_{m\to\infty} T_{n,m} = \infty,
    \end{equation*}
    and
    \begin{equation*}
    	    \Lambda_{s}^{\hat{Z}_{n\Delta_1},i_0} = \Lambda_{T_{n,m}}^{\hat{Z}_{n\Delta_1},i_0}, \quad \forall\,s\in[T_{n,m},T_{n,m+1}).
    \end{equation*}
    Define the stopping time
    \begin{equation*}
    	    \eta_{n,R} := \inf\left\{k\in\mathbb{N}_+ \,:\, T_{n,k} \geq R\right\}.
    \end{equation*}
    By construction, the interval $[0,R]$ can be partitioned as
    \begin{equation*}
    	    [0,R] = \bigcup_{m=0}^{\eta_{n,R}-1} \left[T_{n,m}, \, T_{n,m+1}\wedge R\right].
    \end{equation*}
    We now compute the integral:
    \begin{align*}
    	\int_{0}^{R} b\left(\hat{Z}_{n\Delta_1}, \Lambda_{s}^{\hat{Z}_{n\Delta_1},i_0}\right) \dif s
    	&= \sum_{m=0}^{\eta_{n,R}-1} \int_{T_{n,m}}^{T_{n,m+1}\wedge R} b\left(\hat{Z}_{n\Delta_1}, \Lambda_{s}^{\hat{Z}_{n\Delta_1},i_0}\right) \dif s \\
    	&= \sum_{m=0}^{\eta_{n,R}-1} b\left(\hat{Z}_{n\Delta_1}, \Lambda_{T_{n,m}}^{\hat{Z}_{n\Delta_1},i_0}\right) \int_{T_{n,m}}^{T_{n,m+1}\wedge R} \dif s \\
    	&= \sum_{m=0}^{\eta_{n,R}-1} \left(T_{n,m+1}\wedge R - T_{n,m}\right) b\left(\hat{Z}_{n\Delta_1}, \Lambda_{T_{n,m}}^{\hat{Z}_{n\Delta_1},i_0}\right).
    \end{align*}
    Dividing both sides by $R$ yields the desired identity.
\end{remark}

\subsection{Strong convergence of Algorithm 3}
First, we derive the error bound between the averaged drift \(\bar{b}(x)\) and the continuous-time drift estimator \(\tilde{b}_R(x)\). This error purely originates from the ergodic averaging effect and contains no discretization bias.
\begin{lemma}
Suppose that (H1)-(H3) hold.  Then, for any $p\geq 2$, there exists a constant $C_p>0$ such that for any $R\geq1$,
    \begin{equation*}
        \E|\bar{b}(x)-\tilde{b}_R(x)|^p\leq C_p(1+|x|^{3p/2})\frac{1}{R^{p/2}}.
    \end{equation*}
\end{lemma}
\begin{proof}
By definition,
\[
    \E\left|\bar{b}(x)-\tilde{b}_R(x)\right|^p
    =\frac{1}{R^p}\E\left|\int_{0}^{R}\left(\bar{b}(x)-b(x,{\Lambda}_{s}^{x,i_0})\right)\dif s\right|^p.
\]
Applying the Poisson equation technique as in Lemma \ref{lem:guji}, but without the discretization error term, yields the desired $O(1/R)$ bound. More specifically,
we work on the equivalent CTMC $(\widetilde{\Lambda}_t^{x,i_0})_{t\geq0}$ constructed via Skorokhod's representation \eqref{eq:skorohod}, and consider the Poisson equation \eqref{eq:poi}. Applying It\^o's formula for  $\Phi(x,\widetilde{\Lambda}_t^{x,i_0})$ and rearranging gives
\begin{align*}
	&\quad \int_{0}^{R}\left(\bar{b}(x)-b(x,\widetilde{\Lambda}_{s}^{x,i_0})\right)\dif s
	\\
    &=-\int_{0}^{R}\int_{[0,\infty)}\bigl(\Phi(x,\widetilde{\Lambda}_{s-}^{x,i_0}+h(x,\widetilde{\Lambda}_{s-}^{x,i_0},z))-\Phi(x,\widetilde{\Lambda}_{s-}^{x,i_0})\bigr)\widetilde{N}(\dif s,\dif z)\\
    &\quad +\Phi(x,\widetilde{\Lambda}_{R}^{x,i_0})-\Phi(x,i_0).
\end{align*}
Taking the $p$-th power and applying the $C_r$-inequality, the boundary term satisfies $\E|\Phi(x,\widetilde{\Lambda}_R^{x,i_0})-\Phi(x,i_0)|^p\leq C_p(1+|x|)^p$.
As for the stochastic integral, the calculation is the same as in \eqref{eq:daiwa}. Dividing by $R^p$ and absorbing the boundary term (since $R\geq1$ and $3p/2\geq p$) completes the proof.
\end{proof}

Just like \eqref{EQ:discrete-time multiscale:01}, for computational convenience, we present the continuous-time interpolated version of the aforementioned Algorithm \ref{alg:continuous}, and with a minor abuse of notation, continue to denote it by $\hat{Z}_t$:
\begin{equation}\label{eq:anotherz}
	\begin{split}
		\dif \hat{Z}_t &=\tilde{b}_{R}(\hat{Z}_{t(\Delta_1)})\dif t+\sigma (\hat{Z}_{t(\Delta_1)})\dif W_t\\
		&= \frac{1}{R}\int_0^R b\left(\hat{Z}_{t(\Delta_1)}, \Lambda_{s}^{\hat{Z}_{t(\Delta_1)},i_0}\right)\dif s \dif t +\sigma (\hat{Z}_{t(\Delta_1)})\dif W_t,\quad \hat{Z}_0= x_0,
	\end{split}
\end{equation}
We also note that the moment estimate in Lemma \ref{lem:juguj} remains valid for the continuous-time estimator $\tilde{b}_R$. Using a similar argument as in Lemma \ref{lem:third3}, we obtain the following result.
\begin{lemma}
Suppose that (H1)-(H3) hold.  Then, for any $T>0$, $x_0\in\mathbb{R}^d$, and $p\geq 2$, there exists a constant $C_{x_0,T,p}$ such that
    \begin{align*}
        \mathbb{E}\left(\sup_{0\leq t\leq T}|Z_t-\hat{Z}_t|^p\right)\leq C_{x_0,T,p}\frac{1}{R^{p/2}},
    \end{align*}
    where $\hat{Z}_t$ is defined by \eqref{eq:anotherz}.
\end{lemma}
\begin{proof}
    Writing $Z_t-\hat{Z}_t$ as the sum of a drift integral and a stochastic integral, applying It\^o's formula to $|Z_t-\hat{Z}_t|^p$, and using the Burkholder-Davis-Gundy inequality  and Young's inequality, we decompose the error into three terms:
\begin{align*}
	\mathbb{E}\left(\sup_{0\leq s\leq t}|Z_s-\hat{Z}_s|^p\right)
	&\leq C_{p,T}\int_0^t\E|\bar{b}(Z_{s(\Delta_1)})-\bar{b}(\hat{Z}_{s(\Delta_1)})|^p\dif s\\
	&\quad+C_{p,T}\int_0^t\E|\bar{b}(\hat{Z}_{s(\Delta_1)})-\tilde{b}_R(\hat{Z}_{s(\Delta_1)})|^p\dif s\\
	&\quad+C_{p,T}\int_0^t\E|Z_{s(\Delta_1)}-\hat{Z}_{s(\Delta_1)}|^p\dif s,
\end{align*}
for any $t\leq T$. 
The first and third terms are controlled by the Lipschitz continuity of $\bar{b}$ and $\sigma$. For the second term, the tower property and the previous lemma give
\[
\E|\bar{b}(\hat{Z}_{s(\Delta_1)})-\tilde{b}_R(\hat{Z}_{s(\Delta_1)})|^p
\leq C_p\E\left[(1+|\hat{Z}_{s(\Delta_1)}|^{3p/2})\right]\frac{1}{R^{p/2}}
\leq C_{x_0,T,p}\frac{1}{R^{p/2}},
\]
where the last inequality uses the moment bound for $\hat{Z}_t$. The Gr\"onwall inequality then yields the desired estimate.
\end{proof}

Proceeding similarly to the discrete-time case, we obtain the strong convergence results.
\begin{theorem}
Suppose that (H1)-(H3) hold. Then, for any $T>0$, $x_0\in\mathbb{R}^d$, $i_0\in\mathbb{S}$, and $p\geq 2$, there exists a constant $C_{x_0,T,p}$ such that
    \begin{align*}
        \mathbb{E}\left(\sup_{0\leq t\leq T}|X^{\varepsilon}_t-\hat{Z}_t|^p\right)\leq C_{x_0,T,p} \left(\varepsilon^{p/2}+\Delta_1^{p/2}+\frac{1}{R^{p/2}}\right).
    \end{align*}
\end{theorem}
\begin{proof}
    By the $C_r$-inequality,
\[
|X^{\varepsilon}_t-\hat{Z}_t|^p\leq C_p\left(|X^{\varepsilon}_t-\bar{X}_t|^p+|\bar{X}_t-Z_t|^p+|Z_t-\hat{Z}_t|^p\right).
\]
Taking the supremum over $t\in[0,T]$ and expectation, the three terms are bounded respectively by the averaging principle (Lemma \ref{lem:first1}), Lemma \ref{lem:second2}, and the previous lemma, yielding
\[
\mathbb{E}\left(\sup_{0\leq t\leq T}|X^{\varepsilon}_t-\hat{Z}_t|^p\right)
\leq C_{x_0,T,p}\left(\varepsilon^{p/2}+\Delta_1^{p/2}+\frac{1}{R^{p/2}}\right).
\]
The proof is complete.
\end{proof}

\begin{remark}
Compared with the discrete-time scheme, the error bound no longer contains the $\Delta_2$ discretization term. For the same total microscopic simulation time $R$, the continuous-time method achieves strictly higher accuracy. Moreover, by choosing $R \propto 1/\Delta_1$, the overall error achieves $O(\Delta_1)$ convergence.
\end{remark}

\subsection{Numerical experiments of Algorithm 3}
\begin{figure}[H]
	\centering
    \includegraphics[width=0.6\linewidth]{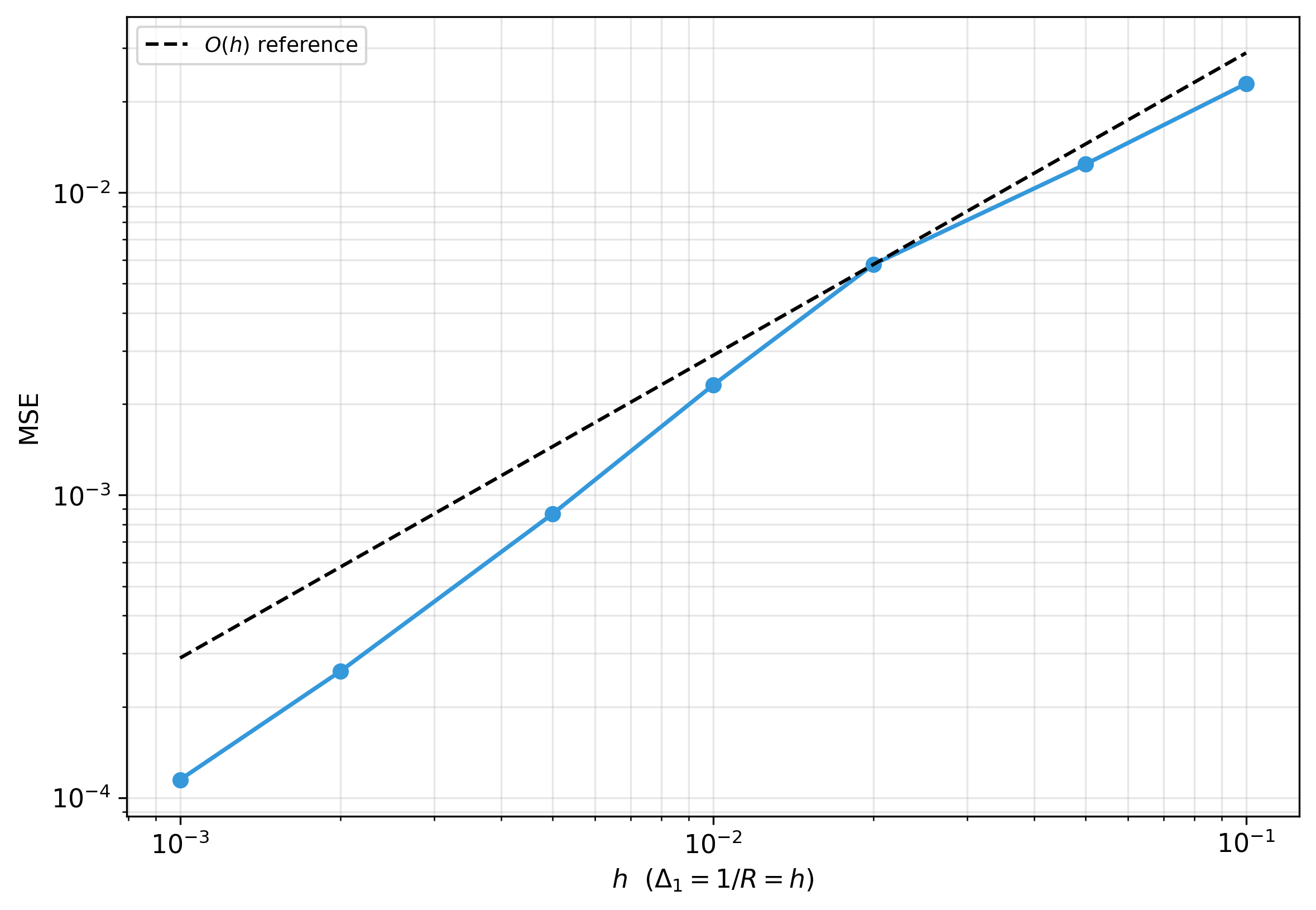}
	\caption{Log-log plot of  MSE between \(\bar{X}_{0.5}^{ref}\) and \(\hat{Z}_{0.5}\) against \(h=(\Delta_1=1/R)\)}
	\label{fig:loglog_mse_h}
\end{figure}
Similarly, by setting $N=100$ in the \textbf{Example (b)}, we further investigate the convergence with respect to \(h=\Delta_1=1/R\). As seen in Figure \ref{fig:loglog_mse_h}, the data points lie approximately along a straight line on the log–log scale, confirming that the MSE decays at first order in $h$.

\begin{table}[H]
  \centering
  \caption{Runtime (s) comparison II with different $N$}
  \begin{tabular}{rccccc}
    \toprule
    $N$ & Algorithm 1 & Algorithm 2 & Algorithm 3 & Speedup (3 vs 1) & Speedup (3 vs 2) \\
    \midrule
    10   & 0.0011 & 0.0039 & 0.0004 &  2.75$\times$ & 9.75$\times$ \\
    50   & 0.0194 & 0.0289 & 0.0070 &  2.77$\times$ &  4.13$\times$ \\
    100  & 0.0394 & 0.0517 & 0.0124 &  3.18$\times$ &  4.17$\times$ \\
    200  & 0.1439 & 0.1480 & 0.0357 &  4.03$\times$ &  4.15$\times$ \\
    500  & 1.2223 & 0.4343 & 0.1046 & 11.69$\times$ &  4.15$\times$ \\
    1000 & 5.4191 & 1.0719 & 0.2155 & 25.15$\times$ &  4.97$\times$ \\
    \bottomrule
  \end{tabular}
  \label{table:ring_compare3}
\end{table}
\begin{figure}[htbp]
	\centering
	\includegraphics[width=0.6\linewidth]{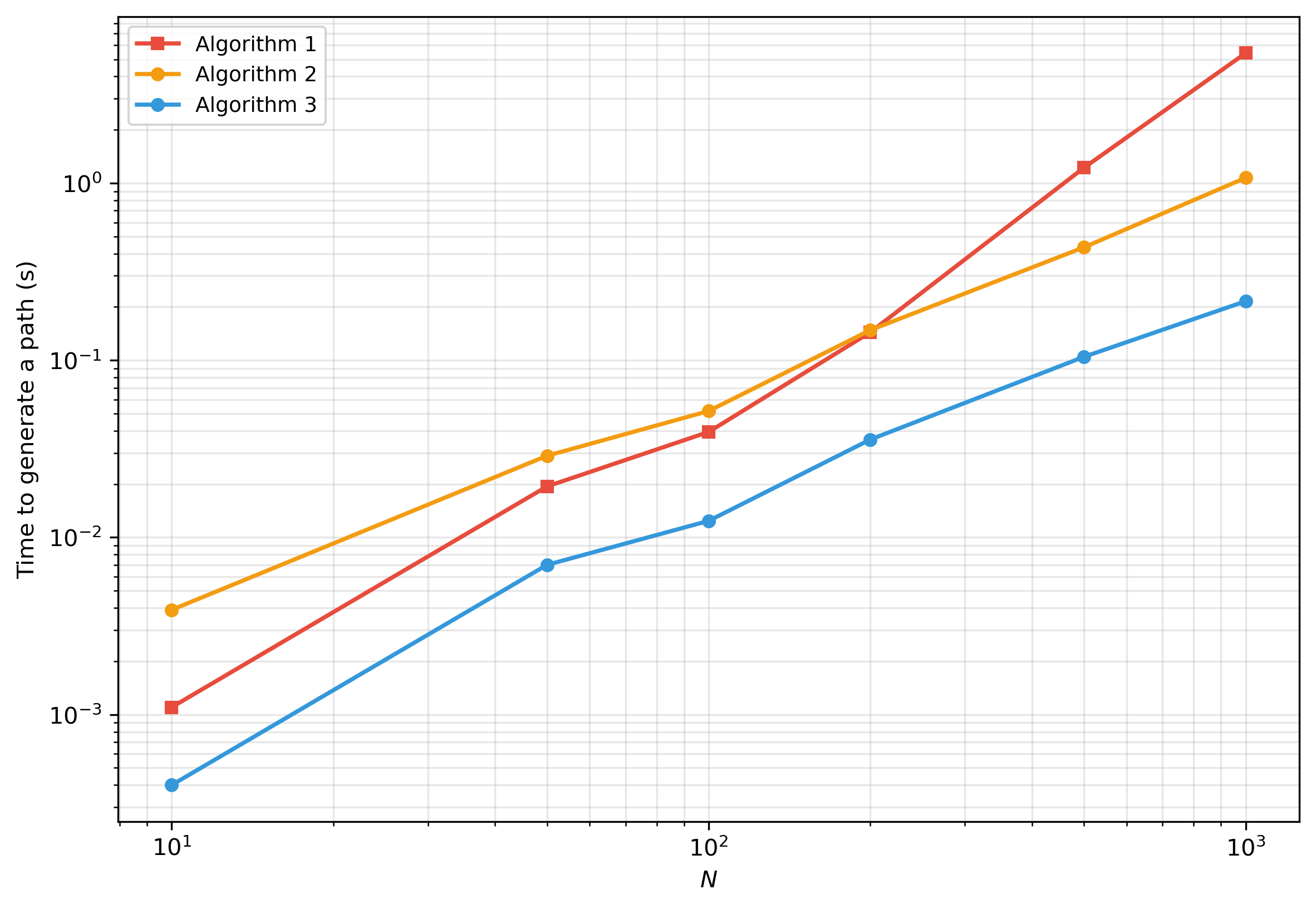}
	\caption{Runtime comparison of Algorithms 1,2, and 3  on log-log scale}
	\label{fig:ring_time_curve}
\end{figure}
Meanwhile, we set the target MSE at \(t^{*}=0.5\) to be \(10^{-4}\) and compare the computational efficiency of Algorithm \ref{alg:hmm-euler}, Algorithm \ref{alg:discrete}, and the proposed  Algorithm \ref{alg:continuous}. Table \ref{table:ring_compare3} and Figure \ref{fig:ring_time_curve} report the runtime corresponding to different state numbers $N$. As $N$ increases, the computational cost of all three schemes grows. Among the three methods,  Algorithm \ref{alg:continuous} consistently yields the shortest runtime both Algorithm \ref{alg:hmm-euler}  and Algorithm \ref{alg:discrete}. Furthermore, the speedup factor of  the Algorithm \ref{alg:continuous} relative to Algorithm \ref{alg:hmm-euler}  increases markedly with $N$ and exceeds \(25\times\) for \(N=1000\). These observations demonstrate that the proposed Algorithm \ref{alg:continuous} achieves superior computational efficiency, particularly for the fast process with large state spaces.


Finally, we conclude with some further remarks on Algorithm \ref{alg:discrete} and Algorithm \ref{alg:continuous}.
For a finite state space, both $Q$ and $P_{\Delta_2}=\mathrm{e}^{\Delta_2Q}$ can be explicitly stored; whereas for an infinite state space, although $Q$, being locally supported, is straightforward to specify, the semigroup $\mathrm{e}^{\Delta_2Q}$ generally admits no closed-form expression and thus cannot be directly sampled from, which makes Algorithm \ref{alg:continuous} more efficient than Algorithm \ref{alg:discrete}. In the following, we present an example of SDE with state-dependent fast switching over an infinite state space to illustrate the effectiveness of Algorithm \ref{alg:continuous}.

\noindent\textbf{Example (c)}  Let the state space be $\mathbb{S} =\mathbb{N}$. Consider the slow-fast coupled system:
\[
\begin{cases}
\dif X_t^\varepsilon = b(X_t^\varepsilon, \Lambda_t^\varepsilon)\dif t + \sigma(X_t^\varepsilon)\dif W_t, \\
\Lambda_t^\varepsilon \text{ is a CTMC with generator } \varepsilon^{-1} Q(x) \text{ frozen at } x = X_t^\varepsilon,
\end{cases}
\]
where $W_t$ is a standard $1$-dimensional Brownian motion. To define the generator $Q(x)$, we first introduce
\[
\rho(x) := \frac{1 + 3/10\sin x}{2 + 3/10\cos x},
\]
and the rate modulation factor
\[
s_i := 2 + 1/10\sin i,\quad i \geq 0.
\]
By construction, $ \rho(x)\in [7/23,13/17]$ uniformly in $x$, and $s_i \in [19/10, 21/10]$ uniformly in $i$.

The fast birth-death chain has transition rates:
\begin{align*}
    \begin{cases}
        q_{i,i+1}(x)=\lambda_i(x) := \rho(x)\, s_i,&i \ge 0;\\
        q_{i,i-1}(x) =\kappa_i(x):= s_{i-1}  &i \ge 1, \text{~with~} q_{0,-1}(x) = 0;\\
        q_{ii}(x) = -(\lambda_i(x)  + \kappa_i(x)),&i \ge 1, \text{~with~}q_{00}(x) = -q_{0,1}(x).
    \end{cases}
\end{align*}
The drift and diffusion coefficients of the slow component are
\[
b(x,i) = -x + \gamma(x)\, \rho(x)^i,\quad \gamma(x) = 3/10 + 1/10\sin(2x),\quad \sigma(x) = 4/5x.
\]
For each frozen $x$, the birth-death chain satisfies the detailed balance condition
\[
\mu_i^x \cdot \kappa_i(x) = \mu_{i-1}^x \cdot \lambda_{i-1}(x),\quad i \ge 1.
\]
Substituting the rates, the factor $s_{i-1}$ cancels out, giving the recursion $\mu_i^x = \rho(x)\, \mu_{i-1}^x$. Normalization yields the unique invariant measure
\[
\mu_i^x = \big(1-\rho(x)\big)\, \rho(x)^i,\quad i \geq 0.
\]
Next, by definition of averaged  coefficient, it follows
\begin{align*}
    \bar{b}(x) &= \sum_{i=0}^{\infty} b(x,i) \, \mu_i^x=-x \sum_{i=0}^{\infty} \mu_i^x + \gamma(x) \sum_{i=0}^{\infty} \rho(x)^i \mu_i^x\\
    &= -x + \gamma(x)(1 - \rho(x)) \sum_{i=0}^{\infty} \rho(x)^{2i}= -x + \frac{\gamma(x)}{1 + \rho(x)}.
\end{align*}

Since $\rho(x)$ is uniformly bounded away from 1, the invariant measure has exponentially decaying tails and the chain is positive recurrent.

We verify that the model satisfies the  assumptions (H1)--(H3).

\noindent(H1)  For Lipschitz continuity in $x$, note that
\[
\left(\gamma(x)\rho(x)^i\right)^{\prime} = \gamma'(x)\rho(x)^i + \gamma(x)\, i\rho(x)^{i-1}\rho'(x).
\]
Since $\sup_{i\ge1} i\rho^{i-1} < \infty$ uniformly in $x$, the derivative is bounded, so $|b(x,i)-b(y,i)| \le C|x-y|$. The  condition $|b(x,i)-b(x,j)| \le C\mathbbm{1}_{\{i\neq j\}}$ holds by boundedness of $\gamma(x)\rho(x)^i$. The diffusion $\sigma(x)=4/5x$ is clearly Lipschitz and of linear growth. Thus (H1) holds.

\noindent(H2) 
(i) Conservativeness holds by construction: off-diagonal rates are non-negative and each row sums to zero; (ii) All birth and death rates are strictly positive, so the chain is irreducible; the unique positive invariant measure is derived above; (iii) Since $s_i \in [19/10, 21/10]$ and $\rho(x) \le \rho_0 < 1$, there exists $\delta>0$ such that $\kappa_i(x) - \lambda_i(x) \ge \delta$ for all $i,x$, giving a uniform drift toward state 0. By the Foster--Lyapunov criterion with $V(i)=\mathrm{e}^{ci}$ for small $c>0$, the chain is uniformly exponentially ergodic:
\[
\sup_{i\in\mathbb{S},\,x\in\mathbb{R}} \| p_{i\cdot}^x(t) - \mu^x \|_{\mathrm{var}} \le C \mathrm{e}^{-\lambda t}.
\]
Thus (H2) holds.

\noindent(H3) Only birth rates depend on $x$, through the Lipschitz function $\rho(x)$. For each row $i$,
\[
\sum_{j\in\mathbb{S}} |q_{ij}(x) - q_{ij}(y)| = s_i\, |\rho(x)-\rho(y)| \le C|x-y|
\]
uniformly in $i$, so $Q(x)$ is Lipschitz in $x$ under the row-sum norm. Total exit rates are uniformly bounded, hence trivially satisfy linear growth. Thus (H3) holds.

  With $\Delta_1 = 1/R$ fixed at $0.1$, $0.05$, $0.02$, $0.01$, $0.005$, and $0.001$, we present the following comparison plots of the trajectories of the averaged equation and Algorithm \ref{alg:continuous}.

\begin{figure}[H]
	\centering
	\includegraphics[width=1\linewidth]{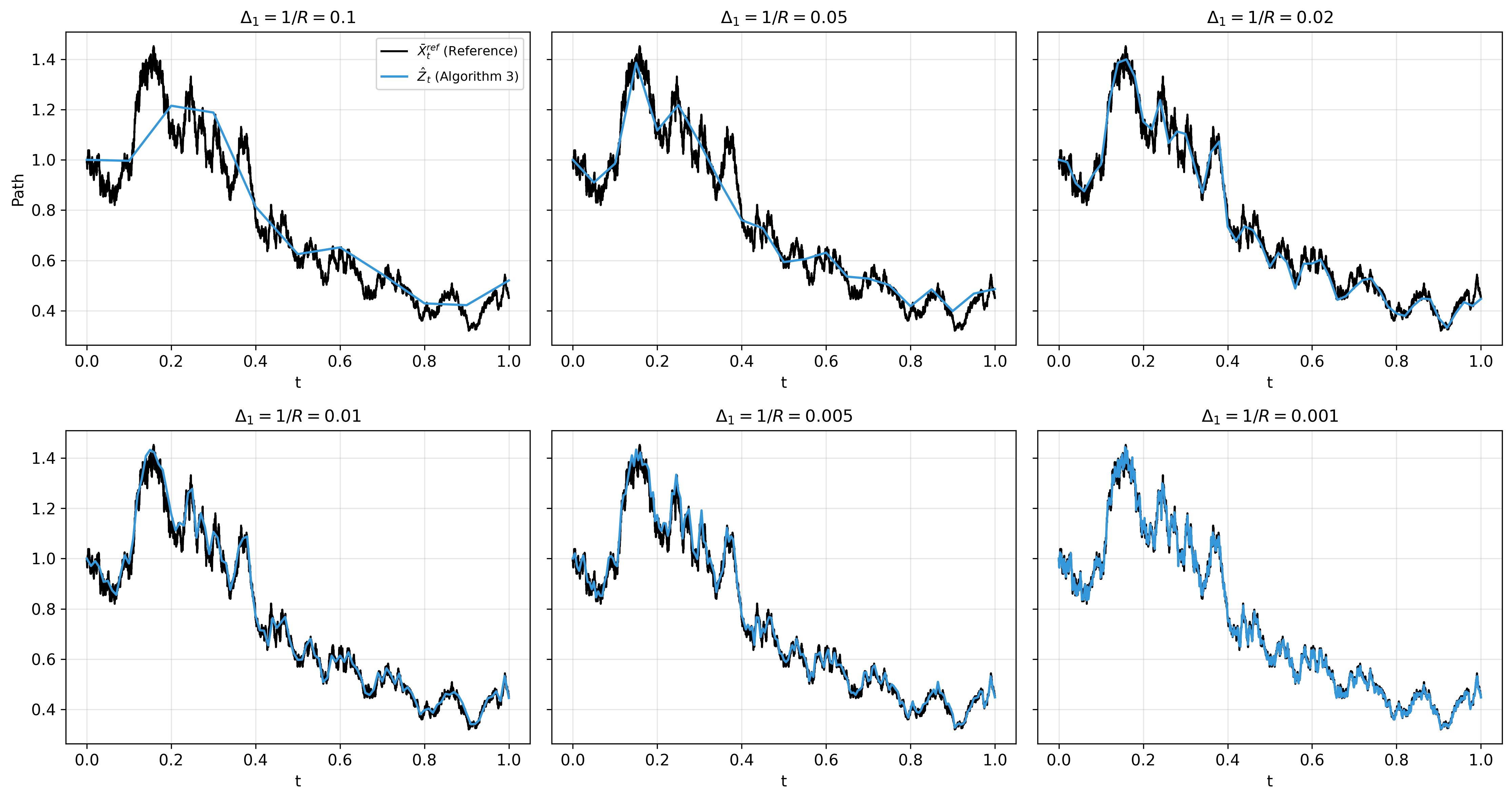}
	\caption{Comparison of sample paths of \(\bar{X}_t^{ref}\) and \(\hat{Z}_t\) as \(\Delta_1 =1/R \downarrow 0\)}
	\label{fig:C_HMMandAveragedequation}
\end{figure}

We further investigate the convergence of Algorithm \ref{alg:continuous}  with respect to \(h=\Delta_1=1/R\). As seen in Figure \ref{fig:loglog_mse_h:02}, the data points lie approximately along a straight line on the log–log scale.

\begin{figure}[H]
	\centering
	\includegraphics[width=0.6\linewidth]{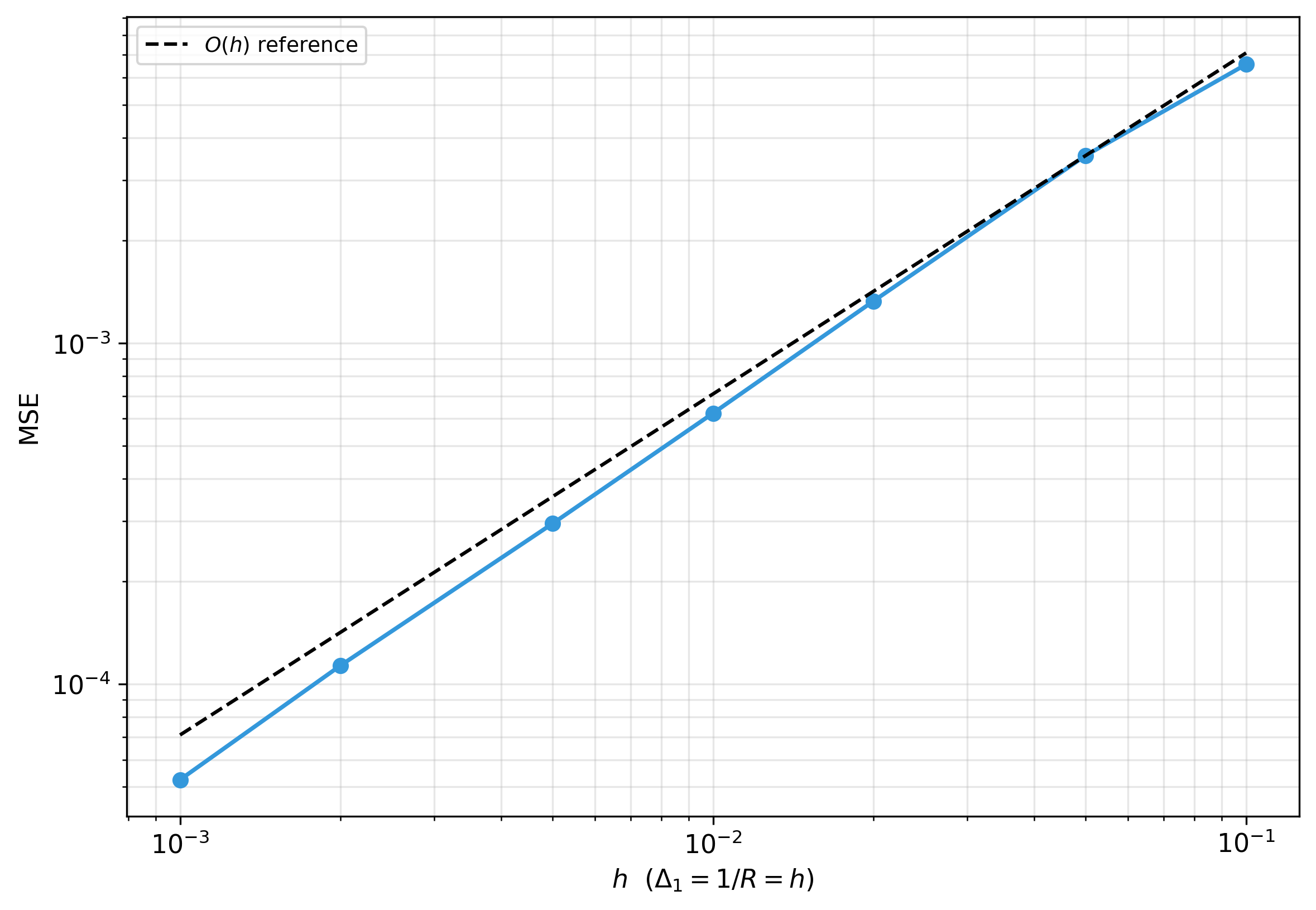}
	\caption{\mbox{Log-log plot of  MSE between \(\bar{X}_{0.5}^{ref}\) and \(\hat{Z}_{0.5}\) against \(h=(\Delta_1=1/R)\)}}
	\label{fig:loglog_mse_h:02}
\end{figure}
\section{Conclusion}\label{SEC:Conclusion}
This paper provides a systematic investigation of SDEs with state-dependent fast switching, focusing on the development of efficient numerical approximation schemes for the slow component $X^{\varepsilon}_t$ of the multiscale stochastic system. To the best of our knowledge, this research has not been adequately discussed in the previous literature.

In contrast to existing studies that mainly consider fast components characterized by diffusion processes, this work focuses on fast components related CTMC. Within this framework, we construct three different numerical approximation theory for the slow component $X^{\varepsilon}_t$. The main achievements of this work can be summarized as follows. 

\begin{enumerate}
    \item We rigorously formulate the numerical approximation problem for SDEs with state-dependent fast switching. Using the HMM framework, we further develop three different numerical algorithms.
    \item Rigorous strong convergence results are established for the three algorithms, and all theoretical derivations are thoroughly verified, thereby guaranteeing the reliability of the conclusions.
    \item Numerical experiments are conducted to validate our results. The results confirm the theoretical convergence rates of the three proposed algorithms, demonstrating the reliability and practicality of our framework.
\end{enumerate}


Based on the techniques used in this paper, future research can be expanded in many aspects. For example, we will relax the restrictive assumptions of the existing framework, including the superlinear growth condition and Hölder continuity condition, and develop the corresponding algorithms to extend the established results to more general multiscale stochastic systems; we will extend the research framework to a broader class of multiscale stochastic systems where the fast process still retains state-dependent switching components, while the slow process is driven by L\'evy processes and fractional Brownian motion; in addition, further research can be conducted on the weak convergence results of the proposed HMM schemes to improve the theoretical completeness of the system.

Overall, this study enriches the numerical approximation theory for SDEs with state-dependent fast switching, provides new methodological insights and rigorous theoretical guarantees for relevant numerical simulations. Moreover, it establishes a foundation for future theoretical developments and practical applications of multiscale stochastic systems.

\vspace{0.3cm}
\textbf{Acknowledgment}. The research of Xiaobin Sun is supported by the NSF of China (Nos.
12271219 and 12671173) and the Priority Academic Program Development of Jiangsu Higher Education Institutions.

\bibliographystyle{abbrv}
\bibliography{Reference_EM}

\end{document}